\documentclass[12pt]{amsart}
\usepackage{a4wide}

\usepackage{enumitem}

\setlist[enumerate]{
	label=\textup{(\arabic*)}
}

\usepackage{mathrsfs}
\usepackage{mathtools}
\usepackage{amsmath}
\usepackage{amssymb}
\usepackage{amsfonts}
\usepackage{graphicx}

\usepackage{amscd}
\usepackage{color}
\usepackage{verbatim}
\ProvideTextCommandDefault{\cprime}{\tprime}

\renewcommand{\ge}{\geqslant}
\renewcommand{\geq}{\geqslant}
\renewcommand{\le}{\leqslant}
\renewcommand{\leq}{\leqslant}

\numberwithin{equation}{section}

\usepackage[hidelinks]{hyperref}
\newtheorem{thm}{Theorem}[section]

\newtheorem{prop}[thm]{Proposition}
\newtheorem{lem}[thm]{Lemma}

\theoremstyle{definition}
\newtheorem{defn}[thm]{Definition}
\theoremstyle{plain}

\theoremstyle{definition}
\newtheorem{rem}[thm]{Remark}

\theoremstyle{plain}

\newcommand{\Widim}{\operatorname{Widim}}
\newcommand{\mdim}{\operatorname{mdim}}

\DeclareMathOperator{\Leb}{Leb}

\def\mb{\mathbb}

\renewcommand{\epsilon}{\varepsilon}

\newtheorem{mainthm}{Main Theorem}

\date{}

\makeatletter
\@namedef{subjclassname@2020}{\textup{2020} Mathematics Subject Classification}
\makeatother

\begin{document}
	\title[Finite-to-One Maps with a Hurewicz Bound]{Finite-to-One Maps into Band-Limited Function Spaces and Cubical Shifts with a Hurewicz Bound}
	
\author{Yonatan Gutman}

\address{\hskip -12pt Yonatan Gutman, Institute of Mathematics, Polish Academy of Sciences, ul.~\'{S}niadeckich 8, 00-656 Warszawa, Poland}
\email{y.gutman@impan.pl}

\author{Xulei Wang}

\address{\hskip -12pt Xulei Wang, Institute of Mathematics, Polish Academy of Sciences, ul.~\'{S}niadeckich 8, 00-656 Warszawa, Poland}
\email{xl.wang.98@outlook.com}


\begin{abstract}
Let $(X,T)$ be a dynamical system with the marker property and
finite mean dimension $m$. We prove that, for every integer $D>m$,
there exists an equivariant continuous map
$f:X\to([0,1]^D)^{\mathbb Z}$ whose fibers contain at most
$\lfloor D/(D-m)\rfloor$ points; in particular, $f$ is an embedding
when $D>2m$. This is the dynamical analogue of the fiber bound
in Hurewicz's classical theorem (1933) for continuous maps into
Euclidean cubes.
For every real number $a>m$, we also construct an equivariant
continuous map from $X$ into the closed unit ball of the real
Bernstein space of bounded continuous functions on $\mathbb R$
whose Fourier transforms, understood in the sense of distributions,
are supported in $[-a/2,a/2]$, with each fiber containing at most
$\lfloor a/(a-m)\rfloor$ points.
The proof adapts the signal-analysis method of
Gutman--Qiao--Tsukamoto (2019) to the finite-to-one setting
and simplifies its analytic and geometric ingredients
for $\mathbb Z$-actions.
\end{abstract}

\maketitle



\section{Introduction}

We begin with Hurewicz's classical finite-to-one theorem \cite[p.~124]{Kuratowskibook}\footnote{Kuratowski's version is stated for maps into $[0,1]^D$.
Since $X$ is compact, the image of any $f\in C(X,\mathbb R^D)$ lies in a cube. Applying the cube-valued density statement after affine normalization, and then rescaling back, gives the corresponding density statement in $C(X,\mathbb R^D)$.
The $G_\delta$ assertion follows from the non-equivariant version of Proposition~\ref{prop:open_K_map}.}.

\begin{thm}[Hurewicz]
	\label{thm:Hurewicz}
		Let $X$ be a compact metrizable space with $\dim X=m$, and let
		$D\in\mathbb N$ satisfy $D>m$. Then the maps
		$f\in C(X,\mathbb R^D)$ satisfying
	\[
	\#f^{-1}(y)
	\le
	\left\lfloor\frac{D}{D-m}\right\rfloor
	\qquad\text{for every }y\in\mathbb R^D
	\]
	form a dense $G_\delta$ subset of $C(X,\mathbb R^D)$, endowed with
	the uniform topology.
\end{thm}

When $D>2m$, we have
\[
\frac{D}{D-m}<2, \quad \text{and thus} \quad \left\lfloor\frac{D}{D-m}\right\rfloor=1.
\]

Hence
	Theorem~\ref{thm:Hurewicz} implies that a generic continuous map $f\colon X\to\mathbb R^D$ is an embedding. Thus, for $\dim X=m$ and
$D>2m$, the classical Menger--N\"obeling embedding theorem  
\cite[Theorem V.2]{Hurewicz-Wallman1941} follows as a special case of
Hurewicz's finite-to-one theorem.

	The purpose of this note is to study a dynamical analogue of this
	finite-to-one phenomenon. In topological dynamics, ordinary dimension is
	replaced by mean dimension. Given a dynamical system $(X,T)$ and
	$D\in\mathbb N$, it is natural to ask whether
	$\operatorname{mdim}(X,T)<D$ implies the existence of an equivariant
	finite-to-one map into the cubic shift $([0,1]^D)^{\mathbb Z}$, and whether
	one can obtain a Hurewicz-type bound on its fibers.

This question was studied by Levin \cite{levin2023}. He proved that if $(X,T)$ is a
\emph{minimal} $\mathbb Z$-system with $\operatorname{mdim}(X,T)=m$, then for every
integer $D>m$, a generic equivariant continuous map
\[
f: X\longrightarrow ([0,1]^D)^{\mathbb Z}
\]
is finite-to-one, and each fiber contains at most $\lfloor\frac{D}{D-m}\rfloor \frac{D}{D-m}$ points. Levin also pointed out that the natural expected estimate for the fiber
cardinality should be the Hurewicz-type bound $\lfloor D/(D-m)\rfloor$.

The main result of this note is to show that this bound is attainable. 
More precisely, we attain the expected Hurewicz bound under the \emph{marker
property}. Hence the result both sharpens Levin's fiber estimate and extends
the conclusion beyond the minimal setting\footnote{A stronger result implying that such maps are generic was obtained in \cite{GLLM26}.}.

\begin{mainthm}\label{mainthm:1}
Let $(X,T)$ be a dynamical system with the marker property. Let $D$ be a natural number such that $m=\mdim(X,T)<D$. Then there exists an equivariant continuous map	$f:X\to([0,1]^D)^\mathbb Z$ such that $f$ is finite-to-one and
\[
\#f^{-1}(y)
\le
\left\lfloor\frac{D}{D-m}\right\rfloor
\qquad\text{for every }y\in([0,1]^D)^\mathbb Z.
\]

	In particular, if $D>2m$, then $f$ is an embedding.
\end{mainthm}

The method used here is based on the signal-analysis approach introduced by
Gutman--Tsukamoto \cite{Gutman-Tsukamoto2020} and further developed by Gutman--Qiao--Tsukamoto \cite{Gutman-Qiao-Tsukamoto2019}. The basic
idea is to encode orbit information into band-limited functions. In the
Gutman--Qiao--Tsukamoto framework, this method was used to prove embedding
theorems for $\mathbb Z^k$-actions with the marker property. Their target
spaces are spaces of band-limited signals, and the bandwidth plays the role of
dimension density.

\subsection{Band-limited signals.}

Throughout this paper, we denote the imaginary unit by $i$. For a rapidly decreasing function $\varphi:\mathbb R\to\mathbb C$, define
\[
\mathcal F(\varphi)(\xi)=\widehat\varphi(\xi)
:=\int_{\mathbb R}\varphi(t)e^{-2\pi it\xi}\,dt,
\qquad
\mathcal F^{-1}(\varphi)(t)
:=\int_{\mathbb R}\varphi(\xi)e^{2\pi it\xi}\,d\xi.
\]
Then $\mathcal F\mathcal F^{-1}\varphi=\mathcal F^{-1}\mathcal F\varphi=\varphi$.
We extend these operators to tempered distributions by duality; see
\cite[Chapter~7]{Schwartz1966}. In particular,
$\mathcal F(e^{2\pi ict})=\delta_c$, where $\delta_c$ is the Dirac
measure at $c\in\mathbb R$.

Following the notation of
\cite[Section~1.2]{Gutman-Tsukamoto2020}, for real numbers $c<d$, let
$V[c,d]$ be the complex Banach space of bounded
continuous functions $\varphi:\mathbb R\to\mathbb C$ satisfying
$\operatorname{supp}\widehat\varphi\subset[c,d]$, equipped with the
supremum-norm, and let
$V^{\mathbb R}[c,d]:=V[c,d]\cap C(\mathbb R,\mathbb R)$ be its real subspace.
For either of these spaces $V$, denote its closed unit ball by
$B_1(V):=\{\varphi\in V:\|\varphi\|_{L^\infty(\mathbb R)}\le1\}$.
Unless otherwise stated, every such unit ball is endowed with the topology of
uniform convergence on compact subsets of $\mathbb R$, induced by
\begin{equation}\label{eq:def_metric_V}
d_V(\varphi,\psi):=
\sum_{n=1}^{\infty}2^{-n}
\|\varphi-\psi\|_{L^\infty([-n,n])}.
\end{equation}
By \cite[Lemma~2.3]{Gutman-Tsukamoto2020}, $B_1(V[c,d])$ is compact in this
topology, and so is its closed real subspace $B_1(V^{\mathbb R}[c,d])$.
Translation defines
a continuous action $\sigma^n(\varphi)(t)=\varphi(t+n)$, and we refer to the
resulting systems as shifts on the band-limited
signals.

We shall use the following real-valued sampling consequence of
\cite[Lemma~2.4]{Gutman-Tsukamoto2020}.

\begin{lem}[Sampling embedding]\label{lem:sampling_embedding}
	Let $D\in\mathbb N$ and $0<a<D$. Then the sampling map $\operatorname{Sam}_D:
	B_1\!\left(V^{\mathbb R}\!\left[-a/2,a/2\right]\right)
	\longrightarrow([-1,1]^D)^{\mathbb Z},$
	\[
	\operatorname{Sam}_D(\varphi):=
	\left(
	\varphi(n),\varphi\left(n+\frac1D\right),\ldots,
	\varphi\left(n+\frac{D-1}{D}\right)
	\right)_{n\in\mathbb Z}
	\]
	is a $\mathbb Z$-equivariant topological embedding.
\end{lem}

Thus we first construct equivariant maps into
$B_1(V^{\mathbb R}[-a/2,a/2])$, where $a$ is the total bandwidth, and then
pass to the cubical shift by sampling at rate $D>a$.

\begin{mainthm}\label{mainthm:2}
	Let $(X,T)$ be a dynamical system with the marker property.
	Set $m=\mdim(X,T)$. If $a$ is a real number with $a>m$, then there exists an equivariant
	continuous map
		$f:X\to B_1\!\left(V^{\mathbb R}\!\left[-a/2,a/2\right]\right)$
	such that
	\[
	\#f^{-1}(\varphi)
	\le
	\left\lfloor \frac{a}{a-m}\right\rfloor
	\qquad
	\text{for every }\varphi\in
	B_1\!\left(V^{\mathbb R}\!\left[-\frac a2,\frac a2\right]\right).
	\]
\end{mainthm}

Main Theorem~\ref{mainthm:1} follows from Main Theorem~\ref{mainthm:2} and Lemma~\ref{lem:sampling_embedding}.

\begin{proof}[Proof: Main Theorem~\ref{mainthm:2} implies Main Theorem~\ref{mainthm:1}]
Since $\mdim(X,T)=m<D$, we can choose a positive number $a$	satisfying
	\[
	m<a<D,\ \text{and}\ \left\lfloor \frac{a}{a-m} \right\rfloor
	=
	\left\lfloor \frac{D}{D-m} \right\rfloor =: K.
	\]
	Main Theorem~\ref{mainthm:2} gives an equivariant continuous map
	$g:X\to B_1(V^{\mathbb R}[-a/2,a/2])$ whose fibers contain at most
	$K$ points. By Lemma~\ref{lem:sampling_embedding},
	$\operatorname{Sam}_D$ is an equivariant embedding. Set
	$f:=\operatorname{Sam}_D\circ g$. Then
		$f:X\to([-1,1]^D)^\mathbb Z$ is an equivariant finite-to-one map satisfying
	$\#f^{-1}(y)\le K$ for every
	$y\in([-1,1]^D)^{\mathbb Z}$. Since the shifts on $([-1,1]^D)^{\mathbb Z}$ and
	$([0,1]^D)^{\mathbb Z}$ are clearly isomorphic, this completes the proof.
\end{proof}

\subsection{Relation to earlier work.}

The proof follows the signal-analysis method of
Gutman--Qiao--Tsukamoto \cite{Gutman-Qiao-Tsukamoto2019},
specialized to \(\mathbb Z\)-actions. In one dimension,
the analytic and geometric arguments become simpler.
This note can therefore serve as an introduction to their
method and to the marker-property strengthening of the embedding
theorem of Gutman--Tsukamoto \cite{Gutman-Tsukamoto2020}.
The main simplification concerns the construction of the
second Voronoi tiling. We replace the homological
zero-persistence argument of
\cite[Section~5.1]{Gutman-Qiao-Tsukamoto2019}
by Rouch\'e's theorem and dispense with the derivative
cutoff used in
\cite[Section~5.4]{Gutman-Qiao-Tsukamoto2019}.
Indeed, our tiling-like function is a nonzero entire
function of one variable, so all its zeros are isolated,
and counting them with multiplicity gives continuous
weighted zero sums, even when zeros merge or split.
These sums define the weights generating the second
Voronoi tiling without requiring quantitative lower
bounds on the derivatives at the zeros.
In several variables, nonzero holomorphic maps may have
nonisolated zeros; moreover, even isolated zeros can
exhibit arbitrarily rapid growth in number, as illustrated
by the Cornalba--Shiffman phenomenon discussed in
\cite[Section~1.4]{Gutman-Qiao-Tsukamoto2019}.
Our construction requires only local persistence of
zeros with multiplicity, rather than global estimates
on their number.
Finally, the Voronoi tiles are intervals, so the geometric
estimates become elementary estimates for the lengths
of intervals and their boundary neighbourhoods.

The paper is organized as follows. Section~\ref{Sec:Pre} reviews mean
dimension and the marker property and collects the finite-dimensional and interpolation tools used later.
Section~\ref{Sec:main_prop} states the main perturbation proposition and
derives Main Theorem~\ref{mainthm:2} from it by a Baire category argument.
Section~\ref{Sec:soc_wel} develops the one-dimensional tiling construction
and the associated tax functions. Finally,
Section~\ref{Sec:pf_main_prop} proves the main proposition.

\section{Preliminaries}\label{Sec:Pre}

This section collects the auxiliary results used in the proof. We first
review mean dimension and the marker property, then record two
finite-dimensional approximation results, and finally recall the
one-dimensional interpolation theorem from
\cite{Gutman-Qiao-Tsukamoto2019}.

Throughout, a dynamical system $(X,T)$ consists of a compact metrizable space $X$ and a homeomorphism $T:X\to X$.
Throughout, $\dim$ denotes Lebesgue covering dimension and $\mathbb N=\{1,2,\ldots\}$. We use \(\#\) for cardinality and \(\Leb\) for
Lebesgue measure on \(\mathbb R\).
\subsection{Review of mean dimension.}

		Here we review the definition of mean dimension
		\cite{Gromov1999, Lindenstrauss-Weiss2000}. Throughout this paper we assume that every simplicial
		complex is finite (i.e. it has only finitely many simplices). 
		
Let $(X,d)$ be a compact metric space and $Y$ a topological space. Let
		$\varepsilon$ be a positive number and $f:X\to Y$ a continuous map. $f$ is
		said to be an $\varepsilon$-embedding if $\operatorname{diam}f^{-1}(y)<\varepsilon$
		for all $y\in Y$. We define the $\varepsilon$-\textbf{width dimension}
		$\Widim_{\varepsilon}(X,d)$ as the minimal integer $n$ such that there exist
		an $n$-dimensional simplicial complex $P$ and an $\varepsilon$-embedding
		$f:X\to P$. It is well-known that Lebesgue covering dimension is given by
		$\dim X=\lim_{\varepsilon\to0}\Widim_\varepsilon(X,d)$.

		Let $(X,T)$ be a dynamical system and let $d$ be a
		compatible metric on $X$. For $N\in\mathbb N$, set
		$d_N(x,y):=\max_{0\le n<N}d(T^nx,T^ny)$. More generally, for
		a finite set $I\subset\mathbb Z$, denote
		$d_I(x,y):=\max_{n\in I}d(T^nx,T^ny)$, with
		$d_\varnothing\equiv 0$.

		We define the \textbf{mean dimension} of the dynamical system $(X,T)$ by
		\[
		\operatorname{mdim}(X,T)
		=
		\lim_{\varepsilon\to 0}
		\left(
		\lim_{N\to\infty}
		\frac{\operatorname{Widim}_{\varepsilon}(X,d_N)}{N}
		\right).
		\]
		
For fixed $\varepsilon$, the inner limit exists because
		$N\mapsto\operatorname{Widim}_{\varepsilon}(X,d_N)$ is subadditive; the
		outer limit exists by monotonicity in $\varepsilon$.
		The mean dimension is a topological invariant of the dynamical system
		$(X,T)$, namely, it is independent of the choice of the distance $d$ (see \cite{Gromov1999}, \S 1.5.1 for more details).

	We say that $(X,T)$ has the \emph{marker property} if for every
	$N\in\mathbb N$ there is an open $U\subset X$ such that
	$U\cap T^{-n}U=\varnothing$ for $0<n<N$ and
	$X=\bigcup_{n\in\mathbb Z}T^{-n}U$.
	
\subsection{Finite-dimensional approximation tools.}
We use the following continuous version of
	\cite[Lemma~2.5]{Gutman-Qiao-Tsukamoto2019}; the cited result gives a simplicial
	map after subdivision, but continuity suffices here.

	\begin{lem}[Approximation lemma, {\cite[Lemma 2.5]{Gutman-Qiao-Tsukamoto2019}}]\label{lem:approx_S}
	Let $\varepsilon,\delta>0$ and $n\in\mathbb N$.
		Let $(X,d)$ be a compact
		metric space and $P$ a simplicial complex. Let $\pi:X\to P$ be an
		$\varepsilon$-embedding and $f:X\to\mathbb{R}^n$ a continuous map
		such that $d(x,y)<\varepsilon$ implies
		$\|f(x)-f(y)\|_\infty<\delta$. Then there is a continuous
		$g:P\to\mathbb R^n$ with
		$\|f(x)-g(\pi(x))\|_\infty<\delta$ for every $x\in X$.
	\end{lem}


In the construction in Section~\ref{Sec:pf_main_prop}, we shall
apply a consequence of Hurewicz's theorem to products of
simplicial complexes. These products carry natural \emph{polytopal}
structures.
We first recall the relevant definitions.

A \emph{polytope} in \(\mathbb R^d\) is the convex hull of a
finite set of points.
We adopt the convention
\(\operatorname{conv}(\varnothing)=\varnothing\).
Following \cite[Definition~2.1]{Ziegler1995},
a subset \(F\subseteq P\) is a \emph{face} of a polytope \(P\)
if there exist a linear functional
\(\ell:\mathbb R^d\to\mathbb R\) and a number \(c\in\mathbb R\)
such that \(\ell(x)\le c\) for every \(x\in P\) and $
F=\{x\in P:\ell(x)=c\}.$
The functional \(\ell\) is allowed to be identically zero.
Thus both \(\varnothing\) and \(P\) are faces of \(P\).

\begin{defn}[Polytopal complex,
    {\cite[Definition~5.1]{Ziegler1995}}]
A \emph{polytopal complex} in \(\mathbb R^d\) is a finite
collection \(\mathcal C\) of polytopes satisfying:
\begin{enumerate}
    \item \(\varnothing\in\mathcal C\);
    \item every face of every member of \(\mathcal C\)
    belongs to \(\mathcal C\);
    \item for all \(P,Q\in\mathcal C\), the intersection
    \(P\cap Q\) is a face of both \(P\) and \(Q\).
\end{enumerate}
\end{defn}

A \emph{polytopal subcomplex} of \(\mathcal C\) is a
subcollection \(\mathcal D\subseteq\mathcal C\) that is itself
a polytopal complex.
Throughout, we identify a polytopal complex \(\mathcal C\)
with its \emph{underlying space}
\(|\mathcal C|=\bigcup_{P\in\mathcal C}P\), equipped with
the subspace topology from \(\mathbb R^d\).
Finite simplicial complexes are polytopal complexes, and
their finite products carry natural polytopal structures.
Products of subcomplexes are subcomplexes of the corresponding
product. A finite union of subcomplexes of a fixed
complex is again a subcomplex.
See \cite{Ziegler1995} for details.

We use the following consequence of Hurewicz's theorem.

\begin{lem}
	\label{lem:con_hure}
	Let $K,D\in\mathbb N$ and let $P$ be a polytopal complex.
	Write $C(P,\mathbb R^D)$ for the space of continuous maps
	$f:P\to\mathbb R^D$ with the norm
	$\|f\|_\infty:=\sup_{x\in P}\|f(x)\|_\infty$. Denote
	$[D]=\{1,\ldots,D\}$, and, for $A\subseteq[D]$, write
	$f_A:=\operatorname{pr}_A\circ f$, where
	$\operatorname{pr}_A:\mathbb R^D\to\mathbb R^A$ is the coordinate
	projection.
	\begin{enumerate}
		\item If $KD>(K+1)\dim P$, then the maps
		$f\in C(P,\mathbb R^D)$ satisfying $\#f^{-1}(y)\le K$ for every
		$y\in\mathbb R^D$ form a dense $G_\delta$ subset of
		$C(P,\mathbb R^D)$.
		
		\item The maps $f\in C(P,\mathbb R^D)$ such that $f_A|_Q: Q\to\mathbb R^A$ has
		fibers of cardinality at most $K$ whenever $Q\subseteq P$ is a
		nonempty polytopal subcomplex and $A\subseteq[D]$ satisfies
		$\#A>\frac{K+1}{K}\dim Q$ form a dense $G_\delta$ subset of
		$C(P,\mathbb R^D)$.
	\end{enumerate}
\end{lem}

\begin{proof}
	For \textup{(1)}, denote $m=\dim P$. The assumption gives $D>m$ and
	$\lfloor D/(D-m)\rfloor\le K$, so the conclusion follows from
	Theorem~\ref{thm:Hurewicz}.
	For \textup{(2)}, fix $Q$ and $A$ as in the statement. The operator
	$\operatorname{res}_{Q,A}:C(P,\mathbb R^D)\to C(Q,\mathbb R^A)$,
	$f\mapsto f_A|_Q$, is bounded and linear. It is surjective by
	the Tietze extension theorem. Hence it is open by the open mapping
	theorem. By \textup{(1)}, the maps in
	$C(Q,\mathbb R^A)$ with fibers of cardinality at most $K$ form a
	dense $G_\delta$ set; its preimage under
	$\operatorname{res}_{Q,A}$ is therefore dense $G_\delta$. There are
	only finitely many pairs $(Q,A)$, so the Baire category theorem
	applied to their intersection proves \textup{(2)}.
\end{proof}

\subsection{Interpolating Functions.}

	This section records the one-dimensional interpolation result from
	\cite[Section 4]{Gutman-Qiao-Tsukamoto2019}. We keep only the notation needed
	later.

For positive $a$ and rational $0<\rho<a$, set
	$\Gamma=(1/\rho)\mathbb Z$ and $\Gamma_1=\mathbb Z+\Gamma$.
Write $\rho=p_0/q_0$ in lowest terms. Then
\[
\Gamma=\frac{q_0}{p_0}\mathbb Z,\qquad
\Gamma_1=\mathbb Z+\Gamma=\frac1{p_0}\mathbb Z,
\]
and $\Gamma_1/\Gamma$ consists of $q_0$ cosets. The admissibility condition below
allows points belonging to the same coset modulo $\Gamma$ to occur without
an additional separation requirement, whereas points from different cosets
must be more than $r$ apart. 
	\begin{defn}\label{def:admissible}
		Let $r>0$.
		\begin{enumerate}[label=\textup{(\arabic*)}, left=0pt]
			\item A subset $\Lambda\subset\Gamma_1$ is called \emph{$r$-admissible}
			if, for any $\lambda_1,\lambda_2\in\Lambda$, either
			$\lambda_1-\lambda_2\in\Gamma$ or $|\lambda_1-\lambda_2|>r$.
			
			\item A function $p:\Gamma_1\to[0,1]$ is called \emph{$r$-admissible}
			if its support $\{\lambda\in\Gamma_1:p(\lambda)>0\}$ is $r$-admissible.
			Let $\mathrm{Adf}_r$ denote the set of all $r$-admissible functions.
		\end{enumerate}
	\end{defn}
The function $p:\Gamma_1\to[0,1]$ below is an
"activation function" indicating which sampling points are active, while
$u:\Gamma_1\to\mathbb R$ contains the values to be interpolated. The
proposition asserts that, provided the active points are sufficiently
sparse, these data can be interpolated by a real-valued function whose
Fourier transform is supported in $[-a/2,a/2]$, with uniform norm control,
translation covariance, and local dependence on the data.
	\begin{prop}[One-dimensional interpolation, {\cite[Proposition 4.7]{Gutman-Qiao-Tsukamoto2019}}]
		\label{prop:prop_psi}
		There are constants $r_0,C_\Psi>0$ and a map $\Psi:\mathrm{Adf}_{r_0}\times
		\ell^\infty(\Gamma_1,\mathbb R)\to V^{\mathbb R}[-a/2,a/2]$
such that, for every
$p\in\mathrm{Adf}_{r_0}$ and
$u\in\ell^\infty(\Gamma_1,\mathbb R)$, the following hold:
		\begin{enumerate}[label=\textup{(\arabic*)}]
			\item If $\lambda\in\Gamma_1$ satisfies $p(\lambda)=1$ then $\Psi(p,u)(\lambda)=u(\lambda)$.
			
			\item $\|\Psi(p,u)\|_{L^\infty(\mathbb{R})}\leq C_\Psi\|u\|_\infty$.
			
			\item For $n\in\mathbb Z$, define $q$ and $v$ by
			$q(\lambda+n)=p(\lambda)$ and $v(\lambda+n)=u(\lambda)$ for
			$\lambda\in\Gamma_1$.
			Then $q\in\mathrm{Adf}_{r_0}$, and
			$\Psi(q,v)(t+n)=\Psi(p,u)(t)$ for each 
			$t\in\mathbb R.$

			\item For every $r,\varepsilon>0$, there are
			$r',\varepsilon'>0$ such that, if $p,q\in\mathrm{Adf}_{r_0}$ and $u,v\in\ell^\infty(\Gamma_1,\mathbb R)$ satisfy
			$\|u\|_\infty,\|v\|_\infty\le1$ and
			$|p(\lambda)-q(\lambda)|,|u(\lambda)-v(\lambda)|<\varepsilon'$ for
			$\lambda\in\Gamma_1\cap[-r',r']$, then
			$|\Psi(p,u)(t)-\Psi(q,v)(t)|<\varepsilon$ for $t\in[-r,r]$.
		\end{enumerate}
	\end{prop}

	\section{Main Proposition}\label{Sec:main_prop}

	In this section we formulate the perturbation proposition which is the main
	technical step of the paper. We then show that it implies Main
	Theorem~\ref{mainthm:2}. The proof is modeled on the usual Baire category
	argument, with the notion of an $\varepsilon$-embedding
	replaced by the following finite-to-one analogue.

	\begin{defn}
For $f:X\to Y$ and $K\ge1$, we call $f$ an $(\varepsilon,K)$-map if every
		$\varepsilon$-separated subset of a fiber of $f$ has at most $K$ points.
		Equivalently, if $E\subset f^{-1}(y)$ and $d(x,x')\ge\varepsilon$ for
		distinct $x,x'\in E$, then $\#E\le K$.
	\end{defn}

	When $K=1$, this is exactly the usual notion of an $\varepsilon$-embedding.
	We will also use the elementary fact that being an $(\varepsilon,K)$-map is an
	open condition in the uniform topology. The argument does not use equivariance,
	but we state it in the equivariant setting needed below.

For dynamical systems $(X,d_X,T)$ and $(Y,d_Y,S)$, let $C_T(X,Y):=\{f\in C(X,Y):f\circ T=S\circ f\}$, endowed with the uniform metric $d_\infty(f,g):=\sup_{x\in X}d_Y(f(x),g(x))$.

	\begin{prop}\label{prop:open_K_map}
		For every $\varepsilon>0$ and $K\in\mathbb N$, the set of maps
		$f\in C_T(X,Y)$ which are $(\varepsilon,K)$-maps is open in $C_T(X,Y)$.
	\end{prop}

\begin{proof}
Fix an $(\varepsilon,K)$-map $f\in C_T(X,Y)$. Write $E_\varepsilon$ as the set of all $(K+1)$-tuples $(x_1,\ldots,x_{K+1})\in X^{K+1}$ with $d_X(x_p,x_q)
\ge\varepsilon$ for all $p\ne q$.
This is a compact set. If $E_\varepsilon=\varnothing$, every $\varepsilon$-separated
subset of $X$ has at most $K$ points, so every map $X\to Y$ is an
$(\varepsilon,K)$-map and there is nothing to prove. Suppose therefore
that $E_\varepsilon\ne\varnothing$, and define
\[
	\Phi_f(x_1,\ldots,x_{K+1})
	:=\max_{1\le p<q\le K+1}d_Y\bigl(f(x_p),f(x_q)\bigr).
\]
Since $f$ is an $(\varepsilon,K)$-map, the images of
$K+1$ pairwise $\varepsilon$-separated points cannot all coincide.
Thus at least one pair has positive $d_Y$-distance, and consequently
$\Phi_f>0$ on $E_\varepsilon$. Thus the function $\Phi_f$ is continuous and strictly
positive on $E_\varepsilon$. Since $E_\varepsilon$ is compact, there is $\eta>0$ such that $\Phi_f\geq \eta$ on $E_\varepsilon$.
			
Let $g\in C_T(X,Y)$ satisfy $d_\infty(f,g)<\eta/3$. Fix
$(x_1,\ldots,x_{K+1})\in E_\varepsilon$, and choose $p<q$ such that
$
d_Y(f(x_p),f(x_q))
=\Phi_f(x_1,\ldots,x_{K+1})
$. 
By the triangle inequality,
$$
\begin{aligned}
d_Y(g(x_p),g(x_q))
&\ge d_Y(f(x_p),f(x_q))
   -d_Y(f(x_p),g(x_p))
   -d_Y(f(x_q),g(x_q))\\
&\ge \Phi_f(x_1,\ldots,x_{K+1})
   -2d_\infty(f,g)>\eta-\frac{2\eta}{3}
=\frac{\eta}{3}>0.
\end{aligned}
$$
Thus the images under $g$ of any $K+1$ pairwise
$\varepsilon$-separated points cannot all coincide. Equivalently, no
fiber of $g$ contains $K+1$ pairwise $\varepsilon$-separated points.
Therefore $g$ is an $(\varepsilon,K)$-map, proving that the set of
$(\varepsilon,K)$-maps is open.
	\end{proof}

	The following proposition is the main technical result of the paper. It is the
	finite-to-one analogue of 
	\cite[Proposition 3.1]{Gutman-Qiao-Tsukamoto2019}. The proof occupies the rest of the paper.

\begin{prop}[Main Proposition]\label{prop:main_prop}
		Let $(X,T)$ be a dynamical system with the marker property,
		let $d$ be a compatible metric on $X$, and set $m=\mdim(X,T)$. If
		$a,\delta>0$, $m<a$, and $f:X\to B_1(V^{\mathbb R}[-a/2,a/2])$ is equivariant and
		continuous, then there is an equivariant continuous map
		$g:X\to B_1(V^{\mathbb R}[-(a+\delta)/2,(a+\delta)/2])$ such that
		\begin{itemize}
			\item $\|g(x)-f(x)\|_{L^\infty(\mathbb{R})} < \delta$ for all $x \in X$.
			\item $g$ is a $(\delta,K)$-map with respect to $d$, where
			$K=\lfloor a/(a-m)\rfloor$.
		\end{itemize}
	\end{prop}
	
	\begin{proof}[Proof of Main Theorem~\ref{mainthm:2}, assuming
		Proposition~\ref{prop:main_prop}]
			Fix a compatible metric $d$ on $X$. For brevity, throughout this proof write
			$B_1(c):=B_1(V^{\mathbb R}[-c/2,c/2])$ for $c>0$.
			Denote $K=\lfloor a/(a-m)\rfloor$ and choose an increasing sequence
			$a_j\uparrow a$ such that $m<a_1$ and
			$\lfloor a_j/(a_j-m)\rfloor=K$ for every $j$. Let $d_a$ be the metric
			\eqref{eq:def_metric_V} on $B_1(a)$ and equip $C_T(X,B_1(a))$ with
			the uniform metric
			$\rho_a(f,g):=\sup_{x\in X}d_a(f(x),g(x))$.
			
			Regard each $C_T(X,B_1(a_j))$ as a subspace of $C_T(X,B_1(a))$, and define
			\[
			\mathcal Z_a:=
			\overline{\bigcup_{j\ge1}C_T(X,B_1(a_j))},
			\]
			where the closure is taken in $C_T(X,B_1(a))$ with respect to
			$\rho_a$. Since $B_1(a)$ is compact, it is complete under $d_a$, and thus $C_T(X,B_1(a))$ is complete.
			Hence $\mathcal Z_a$ is complete and therefore a Baire space.

			For $n\ge1$, let $\mathcal U_n\subseteq\mathcal Z_a$ be the set of
			$(1/n,K)$-maps. It is open by
			Proposition~\ref{prop:open_K_map}. To prove density, fix
			$h\in\mathcal Z_a$ and $r>0$. Choose $j$ and
			$f_0\in C_T(X,B_1(a_j))$ with
			$\rho_a(h,f_0)<r/2$, and then choose
			$0<\delta<\min\{1/n,r/2,a_{j+1}-a_j\}$. Applying
			Proposition~\ref{prop:main_prop} to $f_0$ at bandwidth $a_j$
			gives a $(\delta,K)$-map
			$g\in C_T(X, B_1(a_j+\delta))
			\subseteq C_T(X, B_1(a_{j+1}))$ such that
			$\sup_x\|g(x)-f_0(x)\|_{L^\infty(\mathbb R)}<\delta$. Since
            $$
\rho_a(g,f_0)=\sup_{x\in X}\sum_{k=1}^{\infty}2^{-k}\|g(x)-f_0(x)\|_{L^\infty([-k,k])}\le \sup_{x\in X}\|g(x)-f_0(x)\|_{L^\infty(\mathbb R)}<\delta,
$$	
we have
			$\rho_a(g,h)<\delta+r/2<r$. Moreover, $\delta<1/n$ implies that
			$g$ is a $(1/n,K)$-map. Thus $\mathcal U_n$ is dense in
			$\mathcal Z_a$.
			
			By the Baire category theorem, choose
			$f\in\bigcap_{n\ge1}\mathcal U_n$. Hence every fiber has cardinality at
			most $K$. Indeed, if a fiber contains
			distinct points $x_0,\ldots,x_K$, then
			$\alpha:=\min_{p<q}d(x_p,x_q)>0$. Choosing $n$ with $1/n\le\alpha$
			would contradict $f\in\mathcal U_n$. This ends the proof.
	\end{proof}

\section{Voronoi Tiling and tax functions}\label{Sec:soc_wel}
This section specializes the band-limited tiling construction of
\cite[Section~5]{Gutman-Qiao-Tsukamoto2019} to one dimension. The
stability of its zeros can then be obtained directly from Rouch\'e's theorem.
For the remainder of the paper, fix $(X,T),d,a,\delta,f$ satisfying the
hypotheses of Proposition~\ref{prop:main_prop}, denote $m=\mdim(X,T)$, and use $d_I$ for the metrics induced by $d$.

	\subsection{Band-limited functions associated with tilings.}

    For \(r>0\) and \(u\in\mathbb C\), let \(B_r(u)\) and \(\overline{B}_r(u)\) denote the open and closed disks of radius \(r\) centered at \(u\), respectively. We write \(B_r=B_r(0)\) and \(\overline{B}_r=\overline{B}_r(0)\). If
		$W\subset\mathbb R$, let $\partial_rW$ be the set of $t\in\mathbb R$ for
		which $[t-r,t+r]$ meets both $W$ and $\mathbb R\setminus W$, and set
		$\operatorname{Int}_rW:=W\setminus\partial_rW$. This should not be confused with the ordinary topological interior of $W$, $\operatorname{Int}W$.\footnote{In general $\operatorname{Int}W\neq \operatorname{Int}_0W$.} 

		   As in \cite[Section~5.2]{Gutman-Qiao-Tsukamoto2019}, choose a rapidly
decreasing function $\chi:\mathbb R\to\mathbb R$ such that
$
\operatorname{supp}\widehat{\chi}\subset [-\delta/8,\delta/8]$ and 
$\widehat{\chi}(0)=1$. Thus $\int_{\mathbb R}\chi(t)\,dt=\widehat{\chi}(0)=1$. Define
\begin{equation}\label{eq:K}
    C_{\chi}:=\int_{\mathbb R}|\chi(t)|\,dt<\infty
\end{equation} 
The function $\chi$ extends to an entire function on $\mathbb C$, still denoted by
			$\chi$, and for every integer $\ell\ge0$ there is
			$C_\ell>0$ such that $|\chi(z)|\le C_\ell(1+|z|)^{-\ell}
			e^{(\pi\delta/4)|\operatorname{Im}z|}$ for all
			$z\in\mathbb C$.
			Moreover, $\int_{\mathbb R}\chi(z-t)\,dt=1$ for every
			$z\in\mathbb C$.

	Set $b=a+\delta/2$ and, for $L>1$, define $$\Theta_L(z):=e^{\pi i\,bz}\sin(\pi z/L).$$ Its zero set is
		$L\mathbb Z$.
		
	Let $\mathcal W=\{W_n\}_{n\in\mathbb Z}$ be a tiling of $\mathbb R$ by
	(compact) intervals (some $W_n$ may be empty), namely

	\begin{itemize}
		\item $\mathbb{R}= \bigcup_{n \in \mathbb{Z}} W_n$.
		\item $\operatorname{Int}(W_m)\cap \operatorname{Int}(W_n)=\varnothing$
	 whenever $m\neq n.$
	\end{itemize}
   Consider the entire function
\[
A_n(z):=\int_{W_n}\chi(z-t)\,dt
\]
This is a smoothed indicator of the tile $W_n$. Since the tiles cover
$\mathbb R$ and $\int_{\mathbb R}\chi(z-t)\,dt=1$, the functions
$A_n$ form a smooth partition of unity.
For $x\in\mathbb R$, we may write $A_n(x)=(\chi*1_{W_n})(x)$. Thus
\begin{equation}\label{eq:supp_A_n}
    \operatorname{supp}\widehat{A_n}
\subseteq
\operatorname{supp}\widehat{\chi}
\subseteq
\left[-\frac{\delta}{8},\frac{\delta}{8}\right].
\end{equation}
\begin{defn}
Let $\mathcal W=\{W_n\}_{n\in\mathbb Z}$ be a tiling of $\mathbb R$ as above. For $L > 1$ we define the \textbf{tiling-like band-limited map}
	$\Phi_L(\mathcal{W}) : \mathbb{C}\to \mathbb{C}$ by
	\begin{equation}\label{eq:def_blmap_tiling}
		\Phi_L(\mathcal{W})(z)
		:= \sum_{n \in \mathbb{Z}} \Theta_L(z - n)A_n(z) =\sum_{n \in \mathbb{Z}} \Theta_L(z - n)
		\int_{W_n} \chi(z - t)\, dt.
	\end{equation}
    \end{defn}
		The decay of $\chi$ makes this series locally uniformly
		convergent on $\mathbb C$; in particular, $\Phi_L(\mathcal W)$ is
		an entire function on $\mathbb C$.
				Denote
		\begin{equation}\label{eq:def_theta_L}\Omega:=\{z\in\mathbb C:|\operatorname{Im}z|\le1\},\quad 
		\theta_L:=\inf\bigl\{|\Theta_L(z)|:z\in\Omega,
		\ |z-Ln|\ge \frac{1}{4}\text{ for all }n\in\mathbb Z\bigr\}>0.
		\end{equation}
        The quantity $\theta_L$ is positive because the zero set of
$\Theta_L$ is exactly $L\mathbb Z$. For $z=x+iy\in\Omega$, we have
\begin{equation}\label{eq:theta_modulus}
|\Theta_L(x+iy)|^2
=e^{-2\pi b y}
\left(
\sin^2\!\left(\frac{\pi x}{L}\right)
+\sinh^2\!\left(\frac{\pi y}{L}\right)
\right).
\end{equation}
Hence $\Theta_L$ can vanish only at the points $Ln$. The condition
$|z-Ln|\ge 1/4$ for every $n\in\mathbb Z$ keeps $z$ uniformly away from
these zeros. Moreover,
\eqref{eq:theta_modulus} is $L$-periodic in the real variable $x$.
Therefore it suffices to minimize it on a compact region.

\begin{lem}[Persistence of zeros of $\Theta_L$]\label{lem:zero_Theta}
			Let $\varphi:\mathbb C\to\mathbb C$ be holomorphic, let
			$S\subset\mathbb C$, and suppose
			$|\varphi(z)-\Theta_L(z)|<\theta_L$ for every $z\in S$.
			Then the following hold.
			
		\begin{enumerate}
				\item $\{z\in S\cap\Omega:\varphi(z)=0\}\subset
				\bigcup_{m\in\mathbb Z}\overline{B}_{1/4}(Lm)$.
				
				\item For every $m\in\mathbb Z$ with $\overline{B}_{1/4}(Lm)\subset S$,
				$\varphi$ has one zero in $\overline{B}_{1/4}(Lm)$ and this zero is simple.
			\end{enumerate}
		\end{lem}
		
		\begin{proof}
			For \textup{(1)}, a zero $z\in S\cap\Omega$ outside the
			indicated disks would satisfy
			$|\Theta_L(z)|=|\Theta_L(z)-\varphi(z)|<\theta_L$, a contradiction to \eqref{eq:def_theta_L}. For
			\textup{(2)}, also by \eqref{eq:def_theta_L}, on $\partial \overline{B}_{1/4}(Lm)$ we have
			$|\varphi-\Theta_L|<\theta_L\le|\Theta_L|$. Rouch\'e's theorem shows that
			$\varphi$ and $\Theta_L$ have the same number of zeros in the disk; hence
			$\varphi$ has one zero in the disk, counted with multiplicity.
		\end{proof}
As $\chi$ is rapidly decaying on the real line and $\Theta_L$ is bounded on the strip $\Omega$, one may choose $E=E(L)>0$ so large that
	\begin{equation}\label{eq:def_E_L}
	2\|\Theta_L\|_{L^\infty(\Omega)}
	\sup_{z\in\Omega}\int_{
    \mathbb R\setminus[\operatorname{Re}z-E,\operatorname{Re}z+E]}
	|\chi(z-t)|\,dt<\theta_L.
	\end{equation}
\begin{rem}\label{rem:Phi_L_naming}
    The name \textit{tiling-like band-limited map} for
	$\Phi_L(\mathcal{W})$ is motivated by the fact that by the proof of Lemma~\ref{lem:MP_Phi}\textup{(2)},\textup{(3)} below on the $E$-interior of each tile $W_n$, the function
$\Phi_L(\mathcal W)$ is uniformly $\theta_L$-close to the translated function
$\Theta_L(\,\cdot-n)$.
\end{rem}
\begin{lem}
	[Main properties of $\Phi_L(\mathcal W)$]\label{lem:MP_Phi}
	\indent
	\begin{enumerate}
		\item The restriction of $\Phi_L(\mathcal W)$ to $\mathbb R$ is bounded
		and continuous, with $\|\Phi_L(\mathcal W)\|_{L^\infty(\mathbb R)}\le C_{\chi}$. If
		$L>4/\delta$, the Fourier transform of $\Phi_L(\mathcal W)\big|_{\mathbb{R}}$ is supported in
		$(a/2,a/2+\delta/2)$.
		
		\item Let $E=E(L)$ be the constant chosen in \eqref{eq:def_E_L}. If $n\in\mathbb Z$, $z\in\Omega$,
		$\operatorname{Re}z\in\operatorname{Int}_E W_n$, and
		$\Phi_L(\mathcal W)(z)=0$, then $z\in \overline{B}_{1/4}(n+Lm)$ for some
		$m\in\mathbb Z$.
		
		\item 
        If $n,m\in\mathbb Z$ and
$n+Lm\in\operatorname{Int}_{E+1}W_n$, then
$\Phi_L(\mathcal W)$ has exactly one zero in
$\overline{B}_{1/4}(n+Lm)$, and this zero is simple.
	\end{enumerate}
\end{lem}

\begin{proof}
	The proof follows the one-dimensional specialization of the argument in
	\cite[Section~5.2]
	{Gutman-Qiao-Tsukamoto2019}. Since $\Phi_L(\mathcal W)$ is entire, its
	restriction to $\mathbb R$ is continuous. Moreover,
	$\|\Theta_L\|_{L^\infty(\mathbb R)}=1$, and hence by \eqref{eq:K}, for $t\in\mathbb R$,
	$|\Phi_L(\mathcal W)(t)|
	\le \int_{\mathbb R}|\chi(t-s)|\,ds=C_{\chi}.$
	Note 
    $$
\Theta_L(x-n)
=e^{\pi i b(x-n)}
  \sin\!\left(\frac{\pi(x-n)}{L}\right)
=c_{n,L,+}\,e^{2\pi i\xi_+x}
  +c_{n,L,-}\,e^{2\pi i\xi_-x},
$$  
where $c_{n,L,+},c_{n,L,-}$ depend only on $n$ and $L$, and
$
\xi_\pm=\frac b2\pm\frac1{2L}.
$
Thus from \eqref{eq:supp_A_n} and  \eqref{eq:def_blmap_tiling}, we conclude
	\[
	\operatorname{supp}\mathcal F\bigl(\Phi_L(\mathcal W)|_{\mathbb R}\bigr)
	\subset
	\left[\frac b2-\frac1{2L}-\frac{\delta}{8},
	\frac b2+\frac1{2L}+\frac{\delta}{8}\right].
	\]
	Since $b=a+\delta/2$, this interval is contained in
	$(a/2,a/2+\delta/2)$ whenever $\frac{1}{2L}<\frac{\delta}{8}$. This proves \textup{(1)}.
	
	For \textup{(2)} and \textup{(3)}, suppose that $z\in\Omega$ and
	$\operatorname{Re}z\in\operatorname{Int}_E W_n$. Then
	$[\operatorname{Re}z-E,\operatorname{Re}z+E]\subset W_n$. Since
	$\int_{\mathbb R}\chi(z-t)\,dt\equiv 1$ for $z$ real-valued, the same holds for $z$ complex-valued. We obtain
\[
\begin{aligned}
&|\Phi_L(\mathcal W)(z)-\Theta_L(z-n)|
=|\sum_{j\in\mathbb Z}
\bigl(\Theta_L(z-j)-\Theta_L(z-n)\bigr)
\int_{W_j}\chi(z-t)\,dt|\\
&\le
2\|\Theta_L\|_{L^\infty(\Omega)}
\int_{\mathbb R\setminus W_n}|\chi(z-t)|\,dt
\le
2\|\Theta_L\|_{L^\infty(\Omega)}
\int_{\mathbb R\setminus[\operatorname{Re}z-E,\operatorname{Re}z+E]}
|\chi(z-t)|\,dt
\overset{\eqref{eq:def_E_L}}{<}\theta_L.
\end{aligned}
\]
If $\Phi_L(\mathcal W)(z)=0$, Lemma~\ref{lem:zero_Theta}\textup{(1)},
	applied after translation by $n$, gives
	$z\in \overline{B}_{1/4}(n+Lm)$ for some $m\in\mathbb Z$. This proves
	\textup{(2)}.
	
	Finally, suppose that
	$n+Lm\in\operatorname{Int}_{E+1}W_n$. For every
	$z\in \overline{B}_{1/4}(n+Lm)$, it is easy to see that $z\in\Omega$ and
	$[\operatorname{Re}z-E,\operatorname{Re}z+E]\subset W_n$. Thus $z$ verifies the assumptions of the previous paragraph and therefore the resulting inequality. Using a translated by $n$  version of Lemma~\ref{lem:zero_Theta}\textup{(2)} with $S:=\overline{B}_{1/4}(n+Lm)$ we conclude $\Phi_L(\mathcal W)$ has exactly one zero in 
	$\overline{B}_{1/4}(n+Lm)$ and this zero is simple. This proves \textup{(3)}.
\end{proof}

\subsection{Dynamical Voronoi diagram.}

As in \cite[Section~5.3]{Gutman-Qiao-Tsukamoto2019}, for every
$M\in\mathbb N$, the marker
property and compactness allow us to fix an
integer $N_M\ge M$ and a function $h_M\in C(X,[0,1])$ such that
\[
(\operatorname{supp}h_M)\cap T^j(\operatorname{supp}h_M)=\varnothing
\quad(0<|j|<M),\qquad
X=\bigcup_{|j|<N_M}T^j\{h_M=1\}.
\]
For $x \in X$ consider the following discrete set in $\mathbb{R}^{2}$:
\[
\left\{
\left( n, \frac{1}{h_M(T^n x)} \right)
\;\middle|\;
n \in \mathbb{Z},\; h_M(T^n x) > 0
\right\}.
\]

Define a convex set $V_0^M(x,n)$ for $n \in \mathbb{Z}$ as follows: If $h_M(T^n x) = 0$ then we set $V_0^M(x,n) = \varnothing$. If $h_M(T^n x) > 0$ then we define it as the set of all $u \in \mathbb{R}^{2}$ satisfying
\[
\forall m \in \mathbb{Z} \text{ with } h_M(T^m x) > 0:
\quad
\left| u - \left(n, \frac{1}{h_M(T^n x)} \right) \right|
\le
\left| u - \left(m, \frac{1}{h_M(T^m x)} \right) \right|.
\]

The sets $V_0^M(x,n)$ form a tiling of $\mathbb{R}^{2}$, i.e.  
$\mathbb{R}^{2}
=
\bigcup_{n \in \mathbb{Z}} V_0^M(x,n).
$
Denote $H_M:=(N_M+1)^2$ and take the horizontal section of this tiling at
height $-H_M$:
\[
W_0^M(x,n):=\{t\in\mathbb R:(t,-H_M)\in V_0^M(x,n)\}.
\]
Each $W_0^M(x,n)$ is either empty or a compact interval, and these
intervals cover $\mathbb R$ and have pairwise disjoint interiors.
Hence $\{W_0^M(x,n)\}_{n\in\mathbb Z}$ is a tiling of $\mathbb R$.

For $j,n\in\mathbb Z$, translation of the sites gives
$V_0^M(T^jx,n)=V_0^M(x,n+j)-(j,0)$ and consequently
$W_0^M(T^jx,n)=-j+W_0^M(x,n+j)$. Moreover, if
$\operatorname{Int}W_0^M(x,n)\ne\varnothing$, then
$W_0^M(y,n)$ converges to $W_0^M(x,n)$ in the Hausdorff topology as
$y\to x$.

The next lemma is the one-dimensional version of
\cite[Lemma~5.10]{Gutman-Qiao-Tsukamoto2019}, and we give a simple proof as follows.
	\begin{lem}\label{Lem:bdry_small}
		For any $r > 0$,		\[
		\lim_{M \to \infty}
		\limsup_{R \to \infty}
		\sup_{x \in X}
		\frac{\Leb\left(\bigcup_{n \in \mathbb{Z}}
			\partial_r W_0^M(x,n) \cap [-R,R] \right)}{2R}
		= 0.
		\]
	\end{lem}

\begin{proof}
	Fix $r>0$ and $M\in\mathbb N$, and set
	$C_M:=\sqrt{(N_M+1)^2+(H_M+1)^2}$. For $R>0$, let
	$J_R(x)$ be the set of all $n\in\mathbb Z$ such that
	$W_0^M(x,n)\cap[-R,R]\ne\varnothing$. We first show that
	\begin{equation}\label{eq:inclusion}
	J_R(x)\subset
	\{n\in\mathbb Z:h_M(T^nx)>0\}\cap[-R-C_M,R+C_M].
	\end{equation}
	Let $n\in J_R(x)$ and choose
	$u\in W_0^M(x,n)\cap[-R,R]$. Then $h_M(T^nx)>0$. Choose an integer
	$p$ nearest to $u$. Since
	$X=\bigcup_{|j|<N_M}T^j\{h_M=1\}$, there is an integer
	$j$ with $|j|<N_M$ such that $h_M(T^{p-j}x)=1$. Denote $m=p-j$.
	Then $|u-m|<N_M+1$. Since $(u,-H_M)\in V_0^M(x,n)$, the definition of 
	Voronoi tiling gives
	\[
		|u-n|
		\le
		\left|(u,-H_M)-
		\left(n,\frac{1}{h_M(T^nx)}\right)\right|		\le	|(u,-H_M)-(m,1)|<C_M.
	\]
	Hence $|n|\le R+C_M$, proving \eqref{eq:inclusion}.
	The set $\{n\in\mathbb Z:h_M(T^nx)>0\}$ is $M$-separated by the
	disjointness of the translates of $\operatorname{supp}h_M$. Therefore
	$\#J_R(x)
	\le 2(R+C_M)/M+1.$
	
	Each $W_0^M(x,n)$ is an interval, so its $r$-boundary has measure at
	most $4r$. If $n\notin J_R(x)$, then $W_0^M(x,n)$ does not meet
	$[-R,R]$; hence its $r$-boundary can meet this interval only within
	distance $r$ of $-R$ or $R$. The union of all these remaining
	contributions has measure at most $2r$. It follows that
	\[
		\Leb\left([-R,R]\cap\bigcup_{n\in\mathbb Z}
		\partial_rW_0^M(x,n)\right)
		\le4r\,\#J_R(x)+2r\le
		4r\left(\frac{2(R+C_M)}{M}+1\right)+2r.
	\]
	Since $C_M$ depends only on $M$, we have
	\[
	\limsup_{R\to\infty}\sup_{x\in X}
	\frac{\Leb\left([-R,R]\cap\bigcup_n\partial_rW_0^M(x,n)\right)}{2R}
	\le\frac{4r}{M}.
	\]
	Letting $M\to\infty$ completes the proof.
\end{proof}

\subsection{Tiling and tax functions from tiling-like band-limited maps.}\label{subsection:tiling_band_maps}

Fix $A,L_0>0$ and an integer $L=L(A,L_0)$ satisfying
	\begin{equation}\label{eq:L_large}
	L>1000(A+1)(L_0+2).
	\end{equation}
	Let $E=E(L)$ be the constant chosen in \eqref{eq:def_E_L}. By
	Lemma~\ref{Lem:bdry_small}, choose $M=M(L_0,L,A)\in\mathbb N$ so that
	$W_0(x,n):=W_0^M(x,n)$ satisfies
	\begin{equation}\label{eq:cont_M}
		\limsup_{R \to \infty}
		\left(
		\sup_{x \in X}
		\frac{
			\Leb\left(\bigcup_{n \in \mathbb{Z}}
			\partial_{E + 2(L+1) + L_0 + 1} W_0(x,n)\cap [-3R,3R]
			\right)
		}{R}
		\right)
		< \frac{1}{L}.
	\end{equation}
	
	For each $x\in X$, apply the construction $\Phi_L$ in
	\eqref{eq:def_blmap_tiling} to the tiling
	$\{W_0(x,n)\}_{n\in\mathbb Z}$ constructed above, and denote the
	resulting function by 
	$\Phi(x):=\Phi_L(\{W_0(x,n)\}_{n\in\mathbb Z})$.
	  We claim that $\Phi(x)\not\equiv0$. By \eqref{eq:cont_M}, for sufficiently large
$R$, $\Leb\left(
\bigcup_{n\in\mathbb Z}\partial_EW_0(x,n)
\cap[-3R,3R]
\right)<\frac{R}{L}.$
Moreover
$
\Leb\left(
\bigcup_{j\in\mathbb Z}[j-1/4,j+1/4]
\cap[-3R,3R]
\right)
\le 3R+1.
$
Since $[-3R,3R]$ has length $6R$ and $L>1$, these two sets do not cover
$[-3R,3R]$ . Choose
\begin{equation}\label{eq:t}
t\in[-3R,3R]\setminus
\left(
\bigcup_{n\in\mathbb Z}\partial_EW_0(x,n)
\cup
\bigcup_{j\in\mathbb Z}[j-1/4,j+1/4]
\right).
\end{equation}
We conclude there is an $n\in\mathbb Z$
such that
$t\in\operatorname{Int}_E W_0(x,n)$. Assume $\Phi(x)(t)=0$. By Lemma~\ref{lem:MP_Phi}\textup{(2)},  $t\in\bigcup_{j\in\mathbb Z}[n+Lj-1/4,n+Lj+1/4]$. This contradicts however \eqref{eq:t}.
Hence $\Phi(x)\not\equiv0$.	
	
	\begin{lem}	[Basic properties of $\Phi(x)$]
		\label{Lem:prop_Phi}
		The map $\Phi$ has the following properties.
		\begin{enumerate}
			\item 
			$\|\Phi(x)\|_{L^\infty(\mathbb{R})} \le C_{\chi}.$
			
			\item 
			If $L > 4/\delta$, the Fourier transform of $\Phi(x)|_{\mathbb{R}}$ is supported in $(a/2,a/2+\delta/2)$.
						
			\item
			$\Phi(T^n x)(z) = \Phi(x)(z+n)$
			for all $x \in X,\; n \in \mathbb{Z},\; z \in \mathbb{C}.$

			\item
			If a sequence $x_n$ in $X$ converges to $x$ then $\Phi(x_n)$ converges to $\Phi(x)$ uniformly over every compact subset of $\mathbb{C}$.
		\end{enumerate}
	\end{lem}

\begin{proof}
	(1) and (2) immediately follow from Lemma~\ref{lem:MP_Phi} (1). (3) and (4) follow from the equivariance and continuity of the tiles $W_0(x, n).$
\end{proof}

Choose a continuous function $\alpha_1:\mathbb C\to[0,1]$ such that
$\alpha_1\equiv1$ on $\overline{B}_{1/4}(0)$ and
$\operatorname{supp}\alpha_1\subset B_{1/2}(0)$.
Since $\Phi(x)\not\equiv0$, its zeros are isolated. For $x\in X$ and
$n\in\mathbb Z$, define
\[
\nu(x,n):=\min\left\{1,\;
\sum_{\Phi(x)(z)=0}
\operatorname{ord}_z\Phi(x)\,\alpha_1(z-n)
\right\},
\]
where $\operatorname{ord}_z\Phi(x)$ denotes the multiplicity of $z$ as
a zero of $\Phi(x)$ and $\nu(x,n)$ is defined to be zero when the sum is empty. Only the zeros in $\overline{B}_{1/2}(n)$ contribute, so the
sum is finite.
By Lemma~\ref{Lem:prop_Phi}\textup{(4)}, locally uniform convergence of
$\Phi(x_j)$ to $\Phi(x)$ implies local convergence of their zeros, counting multiplicity. Hence the above weighted
sum depends continuously on $x$, and so does $\nu(x,n)$. Moreover,
Lemma~\ref{Lem:prop_Phi}\textup{(3)} gives
$\nu(T^\ell x,n)=\nu(x,n+\ell)$ for all $\ell,n\in\mathbb Z$.

We now consider the Voronoi tiling generated by $\nu$. Let $x \in X$ and consider
\[
\left\{
\left( n, \frac{1}{\nu(x,n)} \right)
\;\middle|\;
n \in \mathbb{Z},\; \nu(x,n) > 0
\right\}.
\]
If $\nu(x,n)=0$, define $V(x,n)=\varnothing$. Otherwise, define $V(x,n)$ as the set of $u \in \mathbb{R}^{2}$ satisfying
\[
\forall m \in \mathbb{Z} \text{ with } \nu(x,m) > 0:\quad
\left| u - \left(n, \frac{1}{\nu(x,n)} \right) \right|
\le
\left| u - \left(m, \frac{1}{\nu(x,m)} \right) \right|.
\]
These form a tiling of $\mathbb{R}^{2}$. Set $W(x,n)
=\{t\in\mathbb R:(t,0)\in V(x,n)\}.$ The sets $W(x,n)$ form a tiling of $\mathbb{R}$ and are also $\mathbb{Z}$-equivariant, namely
\[
W(T^n x, m) = -n + W(x, n+m)\quad \text{for each}\ n, m\in\mb Z, x\in X.
\]

For $j,m\in\mathbb Z$, Lemma~\ref{lem:MP_Phi}\textup{(2),(3)}
		give
		\begin{equation}\label{eq:nu_zeros}
		  		\begin{aligned}
			m\in\operatorname{Int}_{E+\frac{1}{2}}W_0(x,j)\setminus(j+L\mathbb Z)
			&\Longrightarrow \nu(x,m)=0,\\
			m\in(j+L\mathbb Z)\cap\operatorname{Int}_{E+1}W_0(x,j)
			&\Longrightarrow \nu(x,m)=1.
		\end{aligned}
	\end{equation}
Thus, deep inside each sufficiently large old tile $W_0(x,j)$,
the new tiles form a regular array of intervals of length $L$,
centered at points of $j+L\mathbb Z$. The relation between the tilings
$\mathcal W(x)=\{W(x,n)\}_{n\in\mathbb Z}$ and
$\mathcal W_0(x)=\{W_0(x,n)\}_{n\in\mathbb Z}$
is illustrated in Figure~\ref{fig:old-new-tilings}.
The quantitative properties of the new tiles needed below are
recorded in Lemma~\ref{lem:prop_W}.

\begin{figure}[htbp]
    \centering
    \includegraphics[width=\textwidth]{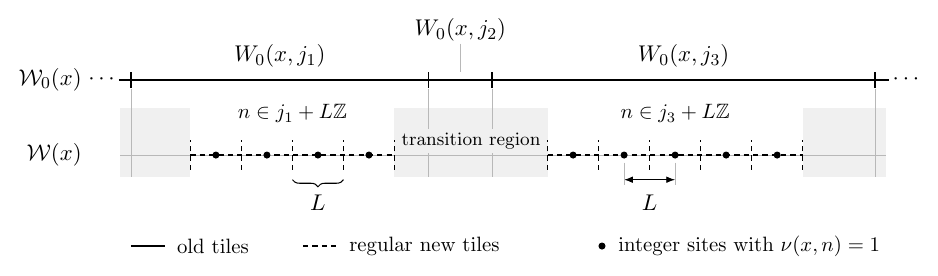}
    \caption{The relation between the old and new tilings.}
    \label{fig:old-new-tilings}
\end{figure}

Each $W(x,n)$ is a (compact) interval of $\mathbb{R}$ (see Lemma below) and depends continuously on $x \in X$ in the Hausdorff topology as in the case of $W_0(x,n)$. Namely, if
$\operatorname{Int}W(x,n)\ne\varnothing$, then
$W(y,n)$ converges to $W(x,n)$ in the Hausdorff topology as $y\to x$.
We write $W^{A,L_0,L}(x,n)$ when emphasizing the parameters. For $x \in X$, set	$\partial(x):=\bigcup_n\partial W(x,n)$ and define
$\operatorname{dist}_{\mathbb R}(t,\partial(x))
:=\inf_{s\in\partial(x)}|t-s|$. Finally, fix a continuous
$\alpha_2:\mathbb R\to[0,3]$ equal to $3$ on $(-\infty,L_0]$ and to
$0$ on $[L_0+1,\infty)$.
	
\begin{lem}[Basic properties of $W(x,n)$]\label{lem:prop_W}
	\indent
	\begin{enumerate}
		\item The diameter of $W(x,n)$ is bounded uniformly in $x \in X$ and $n \in \mathbb{Z}$. Moreover
		\[
		\sup_{x \in X} \sup_{n \in \mathbb{Z}} \sup_{u \in W(x,n)} |n - u| < \infty.
		\]
			
		\item Let $A,L_0,L=L(A,L_0)$  be the parameters fixed in the beginning of Subsection \ref{subsection:tiling_band_maps}. There exists $R_0=R_0(A,L_0,L)>0$ such that for
		all $x\in X$ and $u\in\mathbb Z$
		\[
		A \sum_{|n-u|\le 2R_0} \alpha_2\bigl(\operatorname{dist}_{\mathbb R}(n,\partial(x))\bigr)
		<
		\sum_{|n-u|\le R_0} \Leb(\operatorname{Int}_{L_0} W(x,n)),
		\]
		where both sums are over $n\in\mathbb Z$.
	\end{enumerate}
\end{lem}

	\begin{proof}
The equations \eqref{eq:nu_zeros} are the one-dimensional counterparts of
		\cite[proof of Lemma~5.15; Equation (5.10)]{Gutman-Qiao-Tsukamoto2019}. The remainder
		of that proof uses only these two implications, the Voronoi definition
		of $W(x,n)$, and the sparse-boundary estimate \cite[Equation (5.9)]{Gutman-Qiao-Tsukamoto2019}. In our setting, \eqref{eq:cont_M} is stronger than the $k=1$ case of
		 \cite[Equation (5.9)]{Gutman-Qiao-Tsukamoto2019}, while \eqref{eq:L_large} is precisely the $k=1$ case of \cite[Condition~5.13]{Gutman-Qiao-Tsukamoto2019}. Hence the argument of
		\cite[proof of Lemma~5.15]
		{Gutman-Qiao-Tsukamoto2019}, with the one-dimensional
		modifications, proves both \textup{(1)} and \textup{(2)}.
	\end{proof}
	
	Lemmas 6.3 and ~6.4 of \cite{Gutman-Tsukamoto2020} introduced a \textit{tax system}, which allows large tiles to compensate
	for the impossibility to construct good perturbations over small tiles. We refer the reader to discussion in \cite{Gutman-Tsukamoto2020} before these lemmas.    The paper \cite{Gutman-Qiao-Tsukamoto2019} reused the tax system but applied it to a tiling generated from the zeros of the tiling-like band-limited map $\Phi$, additionally proving that the resulting tax functions are functions of $\Phi$. We now state the properties of the tax functions and discuss
their interpretation after the proof.

    \begin{prop}[Construction of tax functions]\label{prop:weight}
	There exists a continuous map
	\[
	X\longrightarrow \bigl([0,1]^{\mathbb Z}\bigr)^{\mathbb Z},
	\qquad
	x\longmapsto \bigl(w_m(x,n)\bigr)_{m, n\in\mathbb Z},
	\]
    with the following
	properties.
	
	\begin{enumerate}
		\item For all $\ell,m,n\in\mathbb Z$ and $x\in X$,
		$w_m(T^\ell x,n)=w_m(x,n+\ell)$.
		
		\item If $\Phi(x)=\Phi(y)$, then
		$w_m(x,n)=w_m(y,n)$ for all $m,n\in\mathbb Z$.
		
		\item For all $x\in X$ and $n\in\mathbb Z$, we have $w_m(x,n)=0$ for $|m|>R_0$. Moreover,
		\[
		\#\{m\in\mathbb Z:w_m(x,n)>0\}
		<
		1+\frac{\Leb(\operatorname{Int}_{L_0}W(x,n))}{A}.
		\]
		In particular, if $\operatorname{Int}_{L_0} W(x,n) = \varnothing$, then $w_m(x,n) = 0$ for all $m \in \mathbb{Z}$.

		\item If $x\in X$ and
		$\operatorname{dist}_{\mathbb R}(0,\partial(x))\le L_0$, then
		there exists $n\in[-R_0,R_0]\cap\mathbb Z$
         such that
		$w_{-n}(x,n)=1$.
	\end{enumerate}
	
	Here $\partial(x)=\bigcup_{n\in\mathbb Z}\partial W(x,n)$ and
	$R_0=R_0(A,L_0,L)$ is given by
	Lemma~\ref{lem:prop_W}\textup{(2)}. We write
	$w^{A,L_0,L}_m(x,n)$ when emphasizing the dependence on the
	parameters.
\end{prop}

\begin{proof}
We follow the construction in the proof of
\cite[Proposition~5.16]{Gutman-Qiao-Tsukamoto2019}
with \(k=1\), using Lemma~\ref{lem:prop_W}\textup{(2)}
in place of \cite[Lemma~5.15(2)]{Gutman-Qiao-Tsukamoto2019}.
The resulting functions satisfy \(w_m(x,n)=0\) for \(|m|>R_0\).
This gives the required map with the stated properties.
\end{proof}


We interpret \(w_m(x,n)\) as the tax paid by the tile \(W(x,n)\)
to the orbit position \(n+m\). Property \textup{(1)} expresses
equivariance, while \textup{(2)} says that the tax functions depend
only on \(\Phi(x)\). Property \textup{(3)} says that a tile can help only a limited
number of orbit positions, depending on the length of its
\(L_0\)-interior. In particular, a (short) tile with empty
\(L_0\)-interior pays no tax.
Property \textup{(4)} says that if \(0\) is within distance
\(L_0\) of a tile boundary, then there is an index
\(n_0\in[-R_0,R_0]\cap\mathbb Z\) such that the tax function
of \(W(x,n_0)\) takes the maximal value \(1\) at \(m=-n_0\). By \textup{(3)}, this tile must have an \(L_0\)-interior of
positive length.
See Figure~\ref{fig:tax-functions}.

\begin{figure}[htbp]
    \centering
    \includegraphics[width=\textwidth]{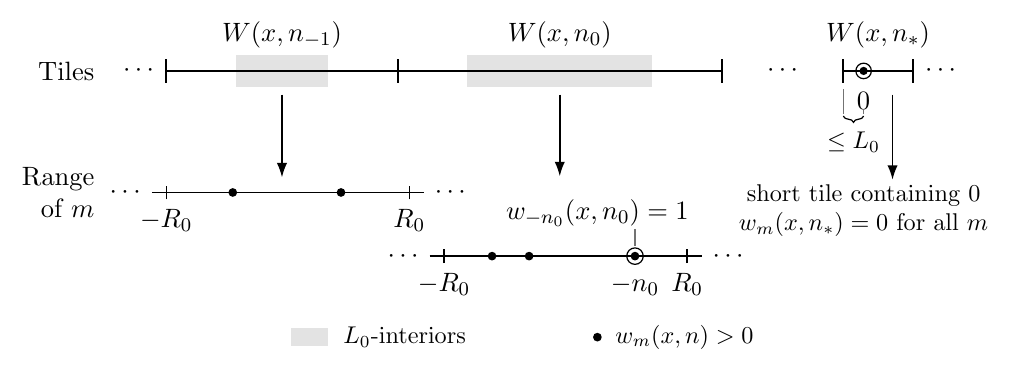}
    \caption{Illustration of the tax functions.}
    \label{fig:tax-functions}
\end{figure}

\section{Proof of the Main Proposition (Proposition \ref{prop:main_prop})}\label{Sec:pf_main_prop}
	
	The organization of the proof follows the one-dimensional form of
	\cite[Section~6]{Gutman-Qiao-Tsukamoto2019}. The two essential modifications
	are the finite-to-one perturbation furnished by	Lemma~\ref{lem:con_hure} and the corresponding fiber-counting argument at the
	end of the proof.

    Throughout this section, we write
\[
K=\left\lfloor\frac{a}{a-m}\right\rfloor
\]
for the uniform Hurewicz bound in
Proposition~\ref{prop:main_prop}.

We first outline the proof.
    We shall construct a $(\delta,K)$-map $g=g_1+g_2$, where $g_1$ is a perturbation of $f$ in the original frequency band and $g_2$ is a small multiple of $\operatorname{Re}\Phi$ in separate frequency bands. The role of $g_2$ is to ensure that points in the same fiber of $g$ have a common Voronoi tiling and the same tax functions.

To compare points in a fiber of \(g\), suppose first
that \(0\) lies well inside a long tile \(W(x,n)\).
For sufficiently large \(N\), mean dimension allows us
to approximate \(N\)-step orbit segments by points of
a finite simplicial complex \(P\), at a scale smaller
than \(\delta\).
We use one point of \(P\) to approximate the orbit on
each time window \([n+j,n+j+N)\), \(j\in N\mathbb Z\),
meeting the tile.
Together, these points describe the orbit over the tile
and provide the coordinates for the finite-dimensional
perturbation. As the tile boundaries move with \(x\), windows enter
and leave the tile.
Simply adding or removing their coordinates in \(P\)
would introduce discontinuities.
We therefore replace \(P\) by its cone \(CP\), giving
each coordinate a weight equal to the fraction of its
window lying inside the tile. The resulting coordinates depend continuously on \(x\),
even when windows enter or leave the tile.

If \(0\) lies near a tile boundary, including when the
tile is short, we use the \emph{tax functions}
in Proposition~\ref{prop:weight}.
They provide a common long tile that pays tax \(1\)
to the orbit position \(0\). 
On this tile, we also record a finite-dimensional
approximation of \(x\) at the same scale.
This allows us to compare the original points even
if the tile does not contain \(0\).

The mean dimension and tax estimates give the dimension
bound required by Lemma~\ref{lem:con_hure}.
We apply this lemma to make a finite-dimensional
perturbation and then use the interpolation result of
Proposition~\ref{prop:prop_psi} to construct \(g_1\).
Together with \(g_2\), this gives the required
\((\delta,K)\)-map \(g=g_1+g_2\).

	\subsection{Choice of parameters.}
    Choose a rational number
	$\rho\in(m,a)$, sufficiently close to $a$, such that
	$\lfloor\rho/(\rho-m)\rfloor=K$. Then
	$m<K\rho/(K+1)$. For this choice of $\rho$, set
	$\Gamma=(1/\rho)\mathbb Z$ and $\Gamma_1=\mathbb Z+\Gamma$, and let
	$r_0,C_\Psi$ and $\Psi$ be as in
	Proposition~\ref{prop:prop_psi}.
	
	Set $\eta=\min\{\delta/4,1/4\}$ and $f_0=(1-\eta)f$. Then
	$\sup_x\|f_0(x)\|_{L^\infty(\mathbb R)}\le 1-\eta$ and
	$\sup_x\|f_0(x)-f(x)\|_{L^\infty(\mathbb R)}\le\eta$. Choose
	$0<\delta'<\min\{1,\eta/(2C_\Psi)\}$, and then choose
	$0<\varepsilon<\delta$ such that
	\begin{equation}\label{eq:f0-uniform-continuity}
d(x,y)<\varepsilon
\quad\Longrightarrow\quad
\|f_0(x)-f_0(y)\|_{L^\infty([0,1])}
<\frac{\delta'}{2}.
\end{equation}
	
	Choose $c_0>1$ so close to one that $c_0m<K\rho/(K+1)$. For fixed
	$\varepsilon$, the sequence
	$\operatorname{Widim}_\varepsilon(X,d_N)/N$ converges to a number not
	exceeding $m$. Hence, by taking a sufficiently large multiple $N$ of the
	denominator of $\rho$, we may find a finite simplicial complex $P$ and an
$\varepsilon$-embedding
$$\Pi:(X,d_N)\to P$$
such that
$\rho N\in\mathbb N$ and
$c_0(\dim P+1)/N<K\rho/(K+1)$.
Indeed, set $\gamma:=K\rho/(c_0(K+1))-m>0$.
For sufficiently large $N$, we may choose $P$ with
$\dim P/N<m+\gamma/2$ and also ensure $1/N<\gamma/2$, whence
$
c_0\frac{\dim P+1}{N}
<c_0(m+\gamma)=\frac{K\rho}{K+1}.
$ Also choose a finite simplicial complex
	$Q$ and an $\varepsilon$-embedding 
    $$\pi:(X,d)\to Q.$$
	
	Let $CP=([0,1]\times P)/(\{0\}\times P)$ be the cone over $P$, and define
	$CQ$ similarly. We write $tp\in CP$ and $tq\in CQ$ for the cone classes and
	denote either cone vertex by $*$. Since $\dim(CP)=\dim P+1$, the numbers
	\[
	c_1:=\frac{K\rho}{K+1}-c_0\frac{\dim(CP)}{N}
	\quad\text{and}\quad
	r_1:=r_0+\frac1\rho
	\]
	satisfy $c_1>0$. Choose $A>2\dim(CQ)/c_1$, and then choose $L_0$ so large that
$$
2L_0-2(r_1+N)>
\max\left\{
\frac{2\dim(CQ)}{c_1},
\frac{2r_1+4N}{c_0-1}
\right\}.
$$
This is possible because $c_1>0$ and $c_0>1$, and it ensures
$L_0>r_1+N$.

If $W$ is a closed interval of length $\ell$ with
$\operatorname{Int}_{L_0}W\ne\varnothing$, then $\ell\ge2L_0$.
Writing $I=\Leb(\operatorname{Int}_{r_1+N}W)$, we obtain
$$
I=\ell-2(r_1+N)\ge2L_0-2(r_1+N)>
\max\left\{
\frac{2\dim(CQ)}{c_1},
\frac{2r_1+4N}{c_0-1}
\right\}.
$$
Thus $2\dim(CQ)<c_1I$, while
$$
\Leb(W\cup\partial_NW)
=\ell+2N=I+2r_1+4N<I+(c_0-1)I=c_0I.
$$
Consequently,
\begin{equation}\label{eq:1559}
    2\dim(CQ)<c_1\Leb(\operatorname{Int}_{r_1+N}W),
    \qquad
    \Leb(W\cup\partial_NW)
    <c_0\Leb(\operatorname{Int}_{r_1+N}W).
\end{equation}

Finally, choose an integer $
L>\max\{4/\delta,1000(A+1)(L_0+2)\}.
$
With these choices of $A,L_0,L$, let $\Phi(x)$, $W(x,n)$,
and $w_m(x,n)$ denote, respectively, the tiling-like band-limited map,
the second Voronoi tiles, and the tax functions
constructed in Section~\ref{Sec:soc_wel}.

For \(x\in X\) and \(n\in\mathbb Z\), set
\begin{equation}\label{eq:def_kappa}
\kappa(x,n):=\#\{m\in\mathbb Z:w_m(x,n)>0\}.
\end{equation}
The following lemma gives the dimension estimate needed
for the finite-dimensional perturbation.
\begin{lem}\label{lem:dim_cont}
		Let \(x\in X\) and \(n\in\mathbb Z\). If $\operatorname{Int}_{L_0}W(x,n)\ne\varnothing$, then
		\begin{equation}\label{eq:1600}
			\frac{\Leb(W(x,n)\cup\partial_NW(x,n))}{N}\dim(CP)
			+\kappa(x,n)\dim(CQ)
			<\frac{K\rho}{K+1}\Leb(\operatorname{Int}_{r_1+N}W(x,n)).
		\end{equation}
	\end{lem}
	
	\begin{proof}
		Write $W=W(x,n)$ and $I_s=\Leb(\operatorname{Int}_sW)$.
By \eqref{eq:1559}, we have
$\Leb(W\cup\partial_NW)<c_0I_{r_1+N}$. By
Proposition~\ref{prop:weight}\textup{(3)} it holds $\kappa(x,n)<1+I_{L_0}/A.$
Multiplying these inequalities by $\dim(CP)/N$ and $\dim(CQ)$,
respectively, and adding, we obtain
\[
\frac{\Leb(W\cup\partial_NW)}{N}\dim(CP)
+\kappa(x,n)\dim(CQ)
<
c_0\frac{\dim(CP)}{N}I_{r_1+N}
+\left(1+\frac{I_{L_0}}{A}\right)\dim(CQ).
\]

      By \eqref{eq:1559} and the choice $c_1A>2\dim(CQ)$, we have
\[
\dim(CQ)<\frac{c_1}{2}I_{r_1+N},
\qquad
\frac{\dim(CQ)}{A}<\frac{c_1}{2}.
\]
Since $L_0>r_1+N$,
$\operatorname{Int}_{L_0}W\subset
\operatorname{Int}_{r_1+N}W$, so $I_{L_0}\le I_{r_1+N}$.
Conclude
\[
\left(1+\frac{I_{L_0}}{A}\right)\dim(CQ)
\le
\dim(CQ)+\frac{\dim(CQ)}{A}I_{r_1+N}
<
\frac{c_1}{2}I_{r_1+N}
+\frac{c_1}{2}I_{r_1+N}
=
c_1I_{r_1+N}.
\]
Therefore by the definition of $c_1$ (LHS) of \eqref{eq:1600} is smaller than
\[
\left(c_0\frac{\dim(CP)}{N}+c_1\right)I_{r_1+N}
=
\frac{K\rho}{K+1}I_{r_1+N}.
\]
\end{proof}
	
	\subsection{Finite-dimensional perturbation.} \label{subsec:finite_dim_perturbation}
	By Lemma~\ref{lem:prop_W}\textup{(1)},
	\[
	D_W:=\sup\{|u-n|:x\in X,\ n\in\mathbb Z,\ u\in W(x,n)\}<\infty.
	\]
		Recall the constant $R_0(A,L_0,L)$ from Lemma~\ref{lem:prop_W}\textup{(2)}. Choose a positive multiple $R$ of $N$ such that
	$R>\max\{R_0(A,L_0,L),D_W\}$. Then
	$-n+W(x,n)\subset[-R,R)$ for every $x\in X$ and $n\in\mathbb Z$. Denote
   \[
	J_N=[-R,R)\cap N\mathbb Z,\qquad J=[-
    R,R)\cap\mathbb Z.
	\]
 and define
	\[
	\mathcal P=(CP)^{J_N},\qquad \mathcal Q=(CQ)^J.
	\]
		For $s\in\mathbb Z_{\ge0}$, let $\mathcal Q(s)$ consist of the points
	$q=(q_m)_{m\in J}$ having at most $s$ coordinates different from $*$. For an
	interval $W\subset[-R,R)$, define
	\[
	\mathcal P(W):=
	\{p=(p_j)_{j\in J_N}\in\mathcal P:p_j=*
	\text{ if }W\cap[j,j+N)=\varnothing\}.
	\]
    
We equip \(CP\) and \(CQ\) with their natural simplicial
structures and \(\mathcal P\) and \(\mathcal Q\) with the
induced product polytopal structures.
Then \(\mathcal P(W)\times\mathcal Q(s)\) is a polytopal
subcomplex of \(\mathcal P\times\mathcal Q\).\footnote{The set \(\mathcal P(W)\) is a product whose
factors are \(CP\) or its vertex subcomplex \(\{*\}\).
It is therefore a subcomplex of \(\mathcal P\).
Similarly, \(\mathcal Q(s)\) is a finite union of products
whose factors are \(CQ\) or \(\{*\}\), and is therefore
a subcomplex of \(\mathcal Q\). Consequently, \(\mathcal P(W)\times\mathcal Q(s)\) is a
polytopal subcomplex of \(\mathcal P\times\mathcal Q\).}

The dimensions of $\mathcal Q(s)$ and $\mathcal P(W)$ are estimated by
	\[
	\dim\mathcal Q(s)\le s\dim(CQ) \qquad \text{and} \qquad
	\dim\mathcal P(W)
	\le\frac{\Leb(W\cup\partial_NW)}{N}\dim(CP).
	\]
	Combining these dimension estimates with Lemma~\ref{lem:dim_cont}, we obtain that if  $\operatorname{Int}_{L_0}W(x,n)\ne\varnothing$,
	\begin{equation}\label{eq:cont_dim_P_Q}
		\dim\bigl(\mathcal P(-n+W(x,n))\times
		\mathcal Q(\kappa(x,n))\bigr)
		<\frac{K\rho}{K+1}\Leb\bigl(\operatorname{Int}_{r_1+N}W(x,n)\bigr).
	\end{equation}
	
	For such an interval $W$, define
	\[
	\Pi_W(x):=\left(
	\frac{\Leb\bigl(W\cap[j,j+N)\bigr)}{N}\,\Pi(T^jx)
	\right)_{j\in J_N}\in\mathcal P(W).
	\]
	Then the map $\Pi_W$ is an $\varepsilon$-embedding for
	$d_{\operatorname{Int}W\cap\mathbb Z}$.	Indeed, suppose that
$\Pi_W(x)=\Pi_W(y)$. If $\Leb(W\cap[j,j+N))>0$ for  $j\in J_N$, then $\Pi(T^jx)=\Pi(T^jy)$. Conclude
$d_N(T^jx,T^jy)<\varepsilon$. As every
$t\in\operatorname{Int}W\cap\mathbb Z$ belongs to a unique block
$[j,j+N)$, $j\in J_N$, whose intersection with $W$ has positive
Lebesgue measure and $t-j\in\{0,\ldots,N-1\}$, it holds $
d(T^tx,T^ty)\le d_N(T^jx,T^jy)<\varepsilon
$.

Suppose
	$d_N(x,y)<\varepsilon$ and let
	$\lambda\in\Gamma\cap[0,N)$. Write $\lambda=k+t$, where
	$k\in\{0,\ldots,N-1\}$ and $t\in[0,1)$. Then
	$d(T^kx,T^ky)<\varepsilon$, \eqref{eq:f0-uniform-continuity} and equivariance give
	\[
	|f_0(x)(\lambda)-f_0(y)(\lambda)|
	=
	|f_0(T^kx)(t)-f_0(T^ky)(t)|
	<\frac{\delta'}{2}.
	\]
	Therefore, applying Lemma~\ref{lem:approx_S} to $\Pi:(X,d_N)\to P$ and the map
	$x\mapsto(f_0(x)(\lambda))_{\lambda\in\Gamma\cap[0,N)}$, we obtain a continuous map
	$F:P\to\mathbb R^{\Gamma\cap[0,N)}$ satisfying
	\[
	|F(\Pi(x))(\lambda)-f_0(x)(\lambda)|<\frac{\delta'}{2}
	\quad(x\in X,\ \lambda\in\Gamma\cap[0,N)).
	\]
	
	Extend $F$ to $CP$ by $F(tp)=tF(p)$. Since $\rho N\in\mathbb N$ and $R$ is
	a multiple of $N$, every element of $\Gamma\cap[-R,R)$ has a unique
	representation $j+\lambda$ with $j\in J_N$ and
	$\lambda\in\Gamma\cap[0,N)$. Define
	$G:\mathcal P\times\mathcal Q\to\mathbb R^{\Gamma\cap[-R,R)}$ by
	$G(p,q)(j+\lambda)=F(p_j)(\lambda)$.
	
	Applying Lemma~\ref{lem:con_hure}\textup{(2)} to 
	$\mathcal P\times\mathcal Q$, with coordinates indexed by
	$\Gamma\cap[-R,R)$, we obtain a continuous map
	$\mathbb G:\mathcal P\times\mathcal Q\to
	\mathbb R^{\Gamma\cap[-R,R)}$ satisfying
	$\|\mathbb G-G\|_\infty<\delta'/2$ and the following property.
	\begin{prop}[Properties of $\mathbb G$]\label{prop:prop_mb_G}
		\indent
		\begin{enumerate}
			\item For every interval $W\subset[-R,R)$, $x\in X$, $q\in\mathcal Q$, and
			$\lambda\in\Gamma\cap\operatorname{Int}_NW$,
			\[|\mathbb G(\Pi_W(x),q)(\lambda)-f_0(x)(\lambda)|<\delta'.\]
						
			\item Let $W\subset[-R,R)$ be an interval, $s\in\mathbb Z_{\ge0}$, and
			$\Lambda\subset\Gamma\cap[-R,R)$. If
			$\#\Lambda>\frac{K+1}{K}\dim(\mathcal P(W)\times\mathcal Q(s))$, then the map
			$(p,q)\mapsto\mathbb G(p,q)|_\Lambda$ on
			$\mathcal P(W)\times\mathcal Q(s)$ has fibers of cardinality at most $K$.
		\end{enumerate}
	\end{prop}
	
	\begin{proof}
    Part \textup{(2)} follows from the choice of \(\mathbb G\)
in Lemma~\ref{lem:con_hure}\textup{(2)}, with
\(Q=\mathcal P(W)\times\mathcal Q(s)\) and \(A=\Lambda\).
		For \textup{(1)}, write $\lambda=j+\lambda_0$ with $j\in J_N$ and
		$\lambda_0\in\Gamma\cap[0,N)$. Since
		$\lambda\in\operatorname{Int}_NW$, the block $[j,j+N)$ is contained in
		$W$, so the $j$-th cone coefficient of $\Pi_W(x)$ is one. Hence
		$G(\Pi_W(x),q)(\lambda)
		=F(\Pi(T^jx))(\lambda_0).$
		Using $f_0(T^jx)(\lambda_0)=f_0(x)(\lambda)$, we obtain
			\[
				|\mathbb G(\Pi_W(x),q)(\lambda)-f_0(x)(\lambda)|
				\le
				\|\mathbb G-G\|_\infty
				+|F(\Pi(T^jx))(\lambda_0)-f_0(T^jx)(\lambda_0)|
				<\delta'.
			\]
			
		This proves \textup{(1)}.
	\end{proof}
	
	Now we are ready to prove the main proposition.
	
	\begin{proof}[Proof of Proposition~\ref{prop:main_prop}]
		Choose continuous functions $\beta_1,\beta_2:\mathbb R\to[0,1]$ such that
		$\beta_1=0$ on $(-\infty,r_0]$ and $\beta_1=1$ on
		$[r_0+1,\infty)$, while $\beta_2=0$ on $(-\infty,N]$ and $\beta_2=1$ on
		$[r_0+N,\infty)$. For $x\in X$ and $n\in\mathbb Z$, denote
		\[
		q_{x,n}:=
		\bigl(w_m(x,n)\cdot\pi(T^{n+m}x)\bigr)_{m\in J}\in\mathcal Q.
		\]
        The vector \(q_{x,n}\) records the auxiliary orbit information
assigned to the tile \(W(x,n)\), with the taxes as weights.
Scalar multiplication is taken in the cone \(CQ\):
a zero tax gives the cone vertex, while a tax of \(1\)
preserves the full coordinate \(\pi(T^{n+m}x)\). Recall that $J=[-R,R)\cap\mathbb Z$.
Since \(w_m(x,n)=0\) for \(|m|>R_0\) and \(R>R_0\),
all indices with positive tax belong to \(J\).
By the definition of \(\kappa\) in \eqref{eq:def_kappa},
at most \(\kappa(x,n)\) coordinates differ from the cone vertex,
so \(q_{x,n}\in\mathcal Q(\kappa(x,n))\).
		
		
		We next define functions $p_x:\Gamma_1\to[0,1]$ and
		$u_x:\Gamma_1\to\mathbb R$. If
		$\lambda\in(n+\Gamma)\cap W(x,n)$ for some $n\in\mathbb Z$, set	
		\[
		\begin{split}
			p_x(\lambda)
			&=\beta_1\bigl(\operatorname{dist}_{\mathbb R}(\lambda,\partial(x))\bigr),\\
			u_x(\lambda)
			&=\beta_2\bigl(\operatorname{dist}_{\mathbb R}(\lambda,\partial(x))\bigr)
			\left\{-f_0(x)(\lambda)+
			\mathbb G\bigl(\Pi_{-n+W(x,n)}(T^nx),q_{x,n}\bigr)(\lambda-n)
			\right\}.
		\end{split}
		\]
		If no such $n$ exists, set $p_x(\lambda)=u_x(\lambda)=0$. If more than one
		$n$ is possible, then $\lambda\in\partial(x)$ and both cutoffs vanish, so
		the definition is independent of the choice of $n$. Moreover, if two points
		in $\operatorname{supp}p_x$ belong to the same tile, their difference lies
		in $\Gamma$; if they belong to different tiles, a tile boundary lies between
		them and their distance is greater than $2r_0$. Thus
		$p_x\in\mathrm{Adf}_{r_0}$.

       We now claim that \(\|u_x\|_{\ell^\infty(\Gamma_1)}\le\delta'\)
for every \(x\in X\).
If \(u_x(\lambda)\ne0\), then
\(\lambda\in(n+\Gamma)\cap W(x,n)\) for some \(n\in\mathbb Z\).
Since \(\beta_2\) vanishes on \((-\infty,N]\), we have
\(\lambda-n\in\Gamma\cap\operatorname{Int}_N(-n+W(x,n))\).
Using \(f_0(T^nx)(\lambda-n)=f_0(x)(\lambda)\)
and \(0\le\beta_2\le1\), we obtain from
Proposition~\ref{prop:prop_mb_G}\textup{(1)} that
\[
|u_x(\lambda)|
\le
\left|
\mathbb G\bigl(\Pi_{-n+W(x,n)}(T^nx),q_{x,n}\bigr)(\lambda-n)
-f_0(T^nx)(\lambda-n)
\right|<\delta'.
\]
This also holds when \(u_x(\lambda)=0\).
Taking the supremum over \(\lambda\in\Gamma_1\) proves the claim.
		
		Set $g_1(x):=f_0(x)+\Psi(p_x,u_x)$. Using the same argument in
		\cite[pp.~1497--1498]{Gutman-Qiao-Tsukamoto2019}, we can show that
		$g_1:X\to V^{\mathbb R}[-a/2,a/2]$ is continuous and equivariant. Moreover,
		Proposition~\ref{prop:prop_psi}\textup{(2)} gives
		$\|g_1(x)-f_0(x)\|_{L^\infty(\mathbb R)}
		\leq C_\Psi\delta'<\eta/2$.
		Define $g_2(x):=\eta\operatorname{Re}\Phi(x)/(2C_{\chi})$ and
		$g:=g_1+g_2$. By Lemma~\ref{Lem:prop_Phi}, the map $g_2$ is continuous
		and equivariant, and
		\[
		\|g_2(x)\|_{L^\infty(\mathbb R)}\le\frac{\eta}{2},
		\qquad
		\operatorname{supp}\widehat{g_2(x)}
		\subset
		\left(-\frac{a+\delta}{2},-\frac a2\right)
		\cup
		\left(\frac a2,\frac{a+\delta}{2}\right).
		\]
		Since $g_1(x)\in V^{\mathbb R}[-a/2,a/2]$, it follows that
		$g(x)\in
		V^{\mathbb R}[-(a+\delta)/2,(a+\delta)/2].$
		Moreover,
		\[
		\|g(x)\|_{L^\infty(\mathbb R)}
		\le 1-\eta+C_\Psi\delta'+\frac{\eta}{2}<1,
		\qquad 
		\|g(x)-f(x)\|_{L^\infty(\mathbb R)}
		<2\eta<\delta.
		\]
		Therefore $g:X\to B_1(V^{\mathbb R}[-(a+\delta)/2,(a+\delta)/2])$ is continuous and equivariant, and it is $\delta$-close to $f$.
		
		It remains to verify that $g$ is a $(\delta,K)$-map with respect to $d$. Fix a fiber
		$g^{-1}(y)$ and let $\mathscr E\subset g^{-1}(y)$ be a
		$\delta$-separated set. There is nothing to prove if
		$\mathscr E=\varnothing$. So in the following, we assume that 
		$\mathscr E\neq\varnothing$.

        We first show that \(g_1(x_1)=g_1(x_2)\) and
\(\Phi(x_1)=\Phi(x_2)\) for all \(x_1,x_2\in\mathscr E\).
In particular, all points of \(\mathscr E\) determine
the same Voronoi tiling and tax functions.
Let \(x_1,x_2\in\mathscr E\).
Since \(g(x_1)=g(x_2)\), we have
\(g_1(x_1)-g_1(x_2)=-(g_2(x_1)-g_2(x_2))\).
The Fourier support of \(g_1(x_1)-g_1(x_2)\)
is contained in \([-a/2,a/2]\), while that of
\(g_2(x_1)-g_2(x_2)\) is contained in
\[
\left(-\frac{a+\delta}{2},-\frac a2\right)
\cup
\left(\frac a2,\frac{a+\delta}{2}\right).
\]
Their intersection is empty, so
\(g_1(x_1)=g_1(x_2)\) and \(g_2(x_1)=g_2(x_2)\).

Set \(h=(\Phi(x_1)-\Phi(x_2))|_{\mathbb R}\).
The equality \(g_2(x_1)=g_2(x_2)\) gives
\(h=-\overline h\).
By Lemma~\ref{Lem:prop_Phi}\textup{(2)},
the Fourier support of \(h\) is contained in
\((a/2,(a+\delta)/2)\), while that of \(\overline h\)
is contained in \((-(a+\delta)/2,-a/2)\),
since complex conjugation reflects Fourier support.
Their intersection is empty, so \(h=0\).
The identity theorem then gives
\(\Phi(x_1)=\Phi(x_2)\).
By the construction in
Subsection~\ref{subsection:tiling_band_maps} and
Proposition~\ref{prop:weight}\textup{(2)},
the tiling and tax functions depend only on \(\Phi(x)\).
Hence \(x_1\) and \(x_2\) have the same tiling
and tax functions.

Fix \(x_*\in\mathscr E\). We write
\(W_n:=W(x_*,n)\) for the tiles of this common tiling,
\(w_m(n):=w_m(x_*,n)\) for the common taxes, and
\(\kappa_n:=\kappa(x_*,n)\) for the number of tax-paying
locations for each tile \(W_n\). Recall that \(\Gamma=(1/\rho)\mathbb Z\).
	For every $n$ with
		$\operatorname{Int}_{L_0}W_n\ne\varnothing$, set $\Lambda_n:=\Gamma\cap
		\bigl(-n+\operatorname{Int}_{r_0+N}W_n\bigr),$ and set
		\[
		\Xi_n( x):=\left(
		\Pi_{-n+W_n}(T^n x),
		\bigl(w_m(n)\cdot\pi(T^{n+m} x)\bigr)_{m\in J}
		\right),
		\qquad  x\in\mathscr E.
		\]
		Then
		$\Xi_n( x)\in
		\mathcal P(-n+W_n)\times\mathcal Q(\kappa_n).$

    Now we show that
\(\mathbb G(\Xi_n(x))|_{\Lambda_n}\)
is independent of \(x\in\mathscr E\).
    Let \(x\in\mathscr E\) and \(\lambda\in\Lambda_n\).
Since \(W(x,n)=W_n\), we have 
$n+\lambda\in(n+\Gamma)\cap
\operatorname{Int}_{r_0+N}W(x,n)\subset\Gamma_1,$
and hence
$\operatorname{dist}_{\mathbb R}(n+\lambda,\partial(x))
\ge r_0+N.$
Since \(N\ge 1\), the definitions of \(p_x\), \(u_x\),
and \(\Xi_n\) give
\[
\begin{aligned}
p_x(n+\lambda)&=1,\\
u_x(n+\lambda)
&=-f_0(x)(n+\lambda)+\mathbb G(\Xi_n(x))(\lambda).
\end{aligned}
\]
The interpolation property in
Proposition~\ref{prop:prop_psi}\textup{(1)} therefore gives
\[
\begin{aligned}
g_1(x)(n+\lambda)
&=f_0(x)(n+\lambda)+\Psi(p_x,u_x)(n+\lambda)\\
&=f_0(x)(n+\lambda)+u_x(n+\lambda)=\mathbb G(\Xi_n(x))(\lambda).
\end{aligned}
\]
Since \(g_1(x)=g_1(x_*)\) for every \(x\in\mathscr E\),
it follows that
\[
\left.\mathbb G(\Xi_n(x))\right|_{\Lambda_n}
=
\left.\mathbb G(\Xi_n(x_*))\right|_{\Lambda_n}
\qquad\text{for all }x\in\mathscr E.
\]

            We next verify the dimension condition in
Proposition~\ref{prop:prop_mb_G}\textup{(2)} for
\(\mathcal P(-n+W_n)\times\mathcal Q(\kappa_n)\)
and the coordinate set \(\Lambda_n\).
For every bounded interval \(I\), we have
\(\#(\Gamma\cap I)\ge\rho\Leb(I)-1\).
Since \(r_1=r_0+1/\rho\), the interval
\(\operatorname{Int}_{r_0+N}W_n\) is longer than
\(\operatorname{Int}_{r_1+N}W_n\) by \(2/\rho\).
Hence
\[
\#\Lambda_n
\ge \rho\Leb(\operatorname{Int}_{r_0+N}W_n)-1
\ge\rho\Leb(\operatorname{Int}_{r_1+N}W_n).
\]
Together with \eqref{eq:cont_dim_P_Q}, this gives
\[
\#\Lambda_n >
\frac{K+1}{K}\dim\bigl(
\mathcal P(-n+W_n)\times\mathcal Q(\kappa_n)
\bigr).
\]
Proposition~\ref{prop:prop_mb_G}\textup{(2)} therefore gives
	\begin{equation}\label{eq:Xi_image_bound}
	\#\Xi_n(\mathscr E)\le K\qquad \text{for each $n$ with} \operatorname{Int}_{L_0}W_n\ne\varnothing.
	\end{equation}
 
Consequently, to prove that $\#\mathscr E\le K$, it is enough to find \(n_0\in\mathbb Z\)
such that \(\operatorname{Int}_{L_0}W_{n_0}\ne\varnothing\)
and \(\Xi_{n_0}\) is injective on \(\mathscr E\).
We distinguish two cases.
First suppose
\(\operatorname{dist}_{\mathbb R}(0,\partial(x_*))>L_0\).
Then \(0\in\operatorname{Int}_{L_0}W_{n_0}\)
for some \(n_0\in\mathbb Z\).
We shall show that \(\Xi_{n_0}\) is injective
on \(\mathscr E\).
Let \(x,x'\in\mathscr E\) satisfy
\(\Xi_{n_0}(x)=\Xi_{n_0}(x')\).
Equality of the first components and the
\(\varepsilon\)-embedding property of
\(\Pi_{-n_0+W_{n_0}}\) imply
\[
d_{(-n_0+\operatorname{Int}W_{n_0})\cap\mathbb Z}
(T^{n_0}x,T^{n_0}x')<\varepsilon.
\]
Since
\(-n_0\in(-n_0+\operatorname{Int}W_{n_0})\cap\mathbb Z\),
it follows that \(d(x,x')<\varepsilon<\delta\).
Hence \(x=x'\) because \(\mathscr E\) is
\(\delta\)-separated. Thus \(\Xi_{n_0}\) is injective
on \(\mathscr E\).

Now suppose
\(\operatorname{dist}_{\mathbb R}(0,\partial(x_*))\le L_0\).
Proposition~\ref{prop:weight}\textup{(4)} gives
\(n_0\in[-R_0,R_0]\cap\mathbb Z\) such that
\(w_{-n_0}(n_0)=1\).
By Proposition~\ref{prop:weight}\textup{(3)},
we have \(\operatorname{Int}_{L_0}W_{n_0}\ne\varnothing\).
Since \(|n_0|\le R_0<R\), the coordinate \(-n_0\)
belongs to \(J=[-R,R)\cap\mathbb Z\).
Let \(x,x'\in\mathscr E\) satisfy
\(\Xi_{n_0}(x)=\Xi_{n_0}(x')\).
Comparing the coordinates indexed by \(-n_0\)
in the second components, we obtain
\[
w_{-n_0}(n_0)\cdot\pi(T^{n_0-n_0}x)
=
w_{-n_0}(n_0)\cdot\pi(T^{n_0-n_0}x').
\]
Since \(w_{-n_0}(n_0)=1\), this gives
\(\pi(x)=\pi(x')\).
The map \(\pi\) is an \(\varepsilon\)-embedding for \(d\),
so \(d(x,x')<\varepsilon<\delta\).
Again, \(x=x'\) because \(\mathscr E\) is
\(\delta\)-separated. Thus \(\Xi_{n_0}\) is injective
on \(\mathscr E\).

In either case, we have found an index \(n_0\) such that
\(\Xi_{n_0}\) is injective on \(\mathscr E\).
Hence  \eqref{eq:Xi_image_bound} gives 
$\#\mathscr E= \#\Xi_{n_0}(\mathscr E)\le K.$ 
Therefore every \(\delta\)-separated subset of every
fiber of \(g\) has at most \(K\) points, so \(g\) is
a \((\delta,K)\)-map. This ends the proof.
	\end{proof}

\subsection*{Acknowledgements}
Y.G. was partially supported by the National Science Centre (Poland)
grant 2020/39/B/ST1/02329.

	\bibliographystyle{amsalpha} 

	\bibliography{main.bib}

\end{document}